\documentclass[a4paper,11pt,reqno]{amsart}

\usepackage{fullpage}

\usepackage{dsfont,amssymb,amsmath, amsfonts, amsthm, mathtools, color, mathalfa, mathrsfs}

\usepackage{hyperref} 

\usepackage{tikz}
\usepackage{calc}
\usepackage{ifthen}

\usepackage[alphabetic,numeric,abbrev,msc-links]{amsrefs}

\DefineSimpleKey{bib}{how}

\usepackage[OT2,T1]{fontenc}
\DeclareSymbolFont{cyrletters}{OT2}{wncyr}{m}{n}
\DeclareMathSymbol{\Sha}{\mathalpha}{cyrletters}{"58}

\allowdisplaybreaks[4]

\numberwithin{equation}{section}

\newtheorem{theorem}{Theorem}[section]
\newtheorem{corollary}[theorem]{Corollary}

\newtheorem{lemma}[theorem]{Lemma}
\newtheorem{proposition}[theorem]{Proposition}

\theoremstyle{remark}
\newtheorem{remark}[theorem]{Remark}
\newtheorem{example}[theorem]{Example}
\theoremstyle{definition}
\newtheorem{definition}[theorem]{Definition}

\newcommand\bp{\begin{proof}}
\newcommand\ep{\end{proof}}

\newcommand\Dhat{{\hat\Delta}}

\newcommand\Ad{\operatorname{Ad}}
\newcommand\Aff{\operatorname{Aff}}
\newcommand\Aut{\operatorname{Aut}}

\newcommand\HS{\operatorname{HS}}

\newcommand\Op{\operatorname{Op}}

\newcommand{\R}{{\mathbb R}}
\newcommand\T{{\mathbb T}}

\newcommand{\F}{{\mathcal F}}
\newcommand{\G}{{\mathcal G}}

\newcommand{\CJ}{\mathcal J}

\newcommand{\CH}{\mathcal H}

\newcommand{\CF}{\mathcal F}

\newcommand{\K}{\mathbb K}
\newcommand{\E}{\mathbb E}

\newcommand{\GL}{\operatorname{GL}}
\newcommand{\Ind}{\operatorname{Ind}}

\newcommand{\vf}{\varphi}

\newlength{\meanderUnitlength}
\newlength{\meanderThickness}
\newcounter{meanderXsep}
\newcounter{meanderYsep}
\newcounter{meanderSum}
\newcounter{meanderTopMax}
\newcounter{meanderBottomMax}
\newcounter{meanderYlevel}
\newcounter{meanderXsize}
\newcounter{meanderYsize}
\newcounter{meanderBulletSize}
\newcounter{meanderSnode}
\newcounter{meanderArcNumber}
\newcounter{meanderloopcount}
\newcounter{meanderA}
\newcounter{meanderB}
\newcounter{meanderM}
\newcounter{meanderH}
\newcounter{meanderarcstyle}
\newcommand{\resetmeanderdefaults}{%
\setlength{\meanderUnitlength}{\unitlength} 
\setlength{\meanderThickness}{0.6pt}
\setcounter{meanderXsep}{24}
\setcounter{meanderYsep}{10}
\setcounter{meanderBulletSize}{3}
\setcounter{meanderarcstyle}{0}
}
\resetmeanderdefaults
\newcommand{\meander}[2]{%
\setcounter{meanderTopMax}{0}
\setcounter{meanderBottomMax}{0}
\setcounter{meanderSum}{0}
\foreach \x in {#1}
	{
	\addtocounter{meanderSum}{\x}
	\ifthenelse{\x > \themeanderTopMax}{\setcounter{meanderTopMax}{\x}}{}
	}
\foreach \x in {#2}
	{
	\ifthenelse{\x > \themeanderBottomMax}{\setcounter{meanderBottomMax}{\x}}{}	
	}
\setcounter{meanderYlevel}{\themeanderYsep*\themeanderBottomMax}
\setcounter{meanderXsize}{\themeanderXsep*\themeanderSum-\themeanderXsep}
\setcounter{meanderYsize}{\themeanderYsep*(\themeanderTopMax +\themeanderBottomMax)}
\setlength{\unitlength}{\meanderUnitlength}
\begin{array}{c}
\begin{picture}(\themeanderXsize,\themeanderYsize)
\linethickness{\meanderThickness}
\ifthenelse{\themeanderBulletSize>0}{\multiput(0,\themeanderYlevel)(\themeanderXsep,0){\themeanderSum}{\circle*{\themeanderBulletSize}}}{}
\setcounter{meanderSnode}{0}
\foreach \x in {#1}
	{
	\addtocounter{meanderSnode}{1}
	\ifthenelse{\isodd{\x}}{\setcounter{meanderArcNumber}{(\x-1)/2}}{\setcounter{meanderArcNumber}{\x/2}}
	\setcounter{meanderloopcount}{0}
	\whiledo{\themeanderloopcount < \themeanderArcNumber}
		{
		\setcounter{meanderA}{(\themeanderSnode+\themeanderloopcount-1)*\themeanderXsep}
		\setcounter{meanderB}{(\themeanderSnode+\x - \themeanderloopcount-2)*\themeanderXsep}
		\setcounter{meanderM}{(\themeanderA+\themeanderB)/2}
		\setcounter{meanderH}{\themeanderYsep*(\x-2*\themeanderloopcount-1)+\themeanderYlevel}
		\ifthenelse{\themeanderarcstyle=1}%
			{
			\qbezier(\themeanderA,\themeanderYlevel)(\themeanderM,\themeanderH)(\themeanderM,\themeanderH)
			\qbezier(\themeanderM,\themeanderH)(\themeanderB,\themeanderYlevel)(\themeanderB,\themeanderYlevel)
			}
			{
			\qbezier(\themeanderA,\themeanderYlevel)(\themeanderA,\themeanderH)(\themeanderM,\themeanderH)
			\qbezier(\themeanderM,\themeanderH)(\themeanderB,\themeanderH)(\themeanderB,\themeanderYlevel)
			}
		\addtocounter{meanderloopcount}{1}
		}
	\addtocounter{meanderSnode}{\x-1}
	}
\setcounter{meanderSnode}{0}
\foreach \x in {#2}
	{
	\addtocounter{meanderSnode}{1}
	\ifthenelse{\isodd{\x}}{\setcounter{meanderArcNumber}{(\x-1)/2}}{\setcounter{meanderArcNumber}{\x/2}}
	\setcounter{meanderloopcount}{0}
	\whiledo{\themeanderloopcount < \themeanderArcNumber}
		{
		\setcounter{meanderA}{(\themeanderSnode+\themeanderloopcount-1)*\themeanderXsep}
		\setcounter{meanderB}{(\themeanderSnode+\x - \themeanderloopcount-2)*\themeanderXsep}
		\setcounter{meanderM}{(\themeanderA+\themeanderB)/2}
		\setcounter{meanderH}{-\themeanderYsep*(\x-2*\themeanderloopcount-1)+\themeanderYlevel}
		\ifthenelse{\themeanderarcstyle=1}%
			{
			\qbezier(\themeanderA,\themeanderYlevel)(\themeanderM,\themeanderH)(\themeanderM,\themeanderH)
			\qbezier(\themeanderM,\themeanderH)(\themeanderB,\themeanderYlevel)(\themeanderB,\themeanderYlevel)
			}
			{
			\qbezier(\themeanderA,\themeanderYlevel)(\themeanderA,\themeanderH)(\themeanderM,\themeanderH)
			\qbezier(\themeanderM,\themeanderH)(\themeanderB,\themeanderH)(\themeanderB,\themeanderYlevel)
			}
		\addtocounter{meanderloopcount}{1}
		}
	\addtocounter{meanderSnode}{\x-1}
	}
\end{picture} \\
\end{array}
}
\newcommand{\meanderstyle}[1]{%
\ifthenelse{\equal{#1}{plain}}{\setcounter{meanderBulletSize}{0}\setcounter{meanderarcstyle}{0}}{}
\ifthenelse{\equal{#1}{default}}{\setcounter{meanderBulletSize}{3}\setcounter{meanderarcstyle}{0}}{}
\ifthenelse{\equal{#1}{line}}{\setcounter{meanderBulletSize}{3}\setcounter{meanderarcstyle}{1}}{}
\ifthenelse{\equal{#1}{plainline}}{\setcounter{meanderBulletSize}{0}\setcounter{meanderarcstyle}{1}}{}
}
\newcommand{\meandersize}[1]{%
\ifthenelse{\equal{#1}{Huge}}{\setlength{\meanderUnitlength}{1.8pt}}{}
\ifthenelse{\equal{#1}{huge}}{\setlength{\meanderUnitlength}{1.6pt}}{}
\ifthenelse{\equal{#1}{Large}}{\setlength{\meanderUnitlength}{1.4pt}}{}
\ifthenelse{\equal{#1}{large}}{\setlength{\meanderUnitlength}{1.2pt}}{}
\ifthenelse{\equal{#1}{default}}{\setlength{\meanderUnitlength}{1.0pt}}{}
\ifthenelse{\equal{#1}{Small}}{\setlength{\meanderUnitlength}{0.8pt}}{}
\ifthenelse{\equal{#1}{small}}{\setlength{\meanderUnitlength}{0.6pt}}{}
\ifthenelse{\equal{#1}{Tiny}}{\setlength{\meanderUnitlength}{0.4pt}}{}
\ifthenelse{\equal{#1}{tiny}}{\setlength{\meanderUnitlength}{0.2pt}}{}
\ifthenelse{\equal{#1}{minuscule}}{\setlength{\meanderUnitlength}{0.1pt}}{}
}
\newcommand{\meanderthickness}[1]{%
\ifthenelse{\equal{#1}{Thin}}{\setlength{\meanderThickness}{0.2pt}}{}
\ifthenelse{\equal{#1}{thin}}{\setlength{\meanderThickness}{0.4pt}}{}
\ifthenelse{\equal{#1}{default}}{\setlength{\meanderThickness}{0.6pt}}{}
\ifthenelse{\equal{#1}{thick}}{\setlength{\meanderThickness}{0.8pt}}{}
\ifthenelse{\equal{#1}{Thick}}{\setlength{\meanderThickness}{1pt}}{}
\ifthenelse{\equal{#1}{wide}}{\setlength{\meanderThickness}{1.2pt}}{}
\ifthenelse{\equal{#1}{Wide}}{\setlength{\meanderThickness}{1.4pt}}{}
}

\begin{document}

\date{April 9, 2026}

\title{Kohn--Nirenberg quantization of the affine group and related examples}

\author[Bieliavsky]{Pierre Bieliavsky}
\email{Pierre.Bieliavsky@uclouvain.be}
\address{Institut de Recherche en Math\'ematique et Physique, Universit\'e Catholique de Louvain, Chemin du Cyclotron, 2, 1348 Louvain-la-Neuve, Belgium}

\author[Gayral]{Victor Gayral}
\email{victor.gayral@univ-reims.fr}
\address{Laboratoire de Math\'ematiques, CNRS UMR 9008, Universit\'e de Reims Champagne-Ardenne,
Moulin de la Housse - BP 1039,
51687 Reims, France}

\author[Neshveyev]{Sergey Neshveyev}
\email{sergeyn@math.uio.no}
\address{Department of Mathematics, University of Oslo,
P.O. Box 1053 Blindern, NO-0316 Oslo, Norway}

\author[Tuset]{Lars Tuset}
\email{larst@oslomet.no}
\address{Department of Computer Science, OsloMet - storbyuniversitetet,
P.O. Box 4 St. Olavs plass, NO-0130 Oslo, Norway}

\begin{abstract}
We show how to construct unitary dual $2$-cocycles  for a class of semidirect products that exhibit many similarities with the affine
group ${\rm Aff}(V)=\GL(V)\ltimes V$ of a finite dimensional vector space over a local skew field. The primary
source of examples comes from Lie groups  whose Lie algebras are Frobenius seaweeds.
The construction builds on our earlier results~\cite{BGNT3} and relies heavily on representation theory and an associated quantization procedure of Kohn--Nirenberg type.

On the technical side, the key point is the observation that any semidirect product  $G=H\ltimes V$ in our class can be presented
as a double crossed product $G=P\bowtie N$ with respect to which the unique square-integrable irreducible representation of $G$ takes a particularly nice form. The Kohn--Nirenberg quantization that we construct is intimately related to a scalar Fourier transform
$\CF\colon L^2(N)\to L^2(P)$ intertwining the left regular representations  of $P$ and $N$ with representations defined  by the dressing transformations.
\end{abstract}

\maketitle

\section*{Introduction}

In this article we continue our project, launched in \cite{BGNT3} and developed further in \cites{GM, BGNT4}, of quantizing certain classes of locally compact groups in the analytic setting.
Given a locally compact group $G$, the aim is to construct a unitary dual $2$-cocycle, that is, a unitary element  $\Omega$ of the group von Neumann
algebra  $W^*(G\times G)$ satisfying the cocycle relation
$$
(\Omega\otimes1)(\Dhat\otimes\iota)(\Omega)=(1\otimes\Omega)(\iota\otimes\Dhat)(\Omega).
$$
Thanks to the seminal  work of De Commer \cite{DC}, it is known that the von Neumann bialgebra $\hat G_\Omega:=(W^*(G),\Omega\Dhat(\cdot)\Omega^*)$
defines a locally compact
quantum group in the sense of Kustermans and Vaes \cites{KV1,KV2}, that is,  $\hat G_\Omega$ comes with invariant weights.

As explained in \cite{BGNT3}, such dual cocycles can be obtained by the following procedure. Assume we are given a square-integrable irreducible projective representation $\pi\colon G\to PU(\CH)$ with $2$-cocycle $\omega\in Z^2(G;\T)$. Assume also that the twisted group von Neumann algebra $W^*(G;\omega)$ is a type I factor and that  we are given a unitary equivariant quantization map
\begin{align*}
\Op\in U\big(L^2(G),\HS(\CH)\big), \quad\text{so that}\quad \pi(g)\Op(f)\pi(g)^*=\Op(\lambda_gf),
\end{align*}
where $\HS(\CH)$ is the Hilbert space of Hilbert--Schmidt operators acting on $\CH$.
Equivalently, this means  that $(B(\CH),\Ad\pi)$ is a $G$-Galois object and that we have a unitary equivalence of representations $\Ad\pi\sim\lambda$.
Under these assumptions, the following  defines a unitary dual $2$-cocycle:
$$
\Omega:= (\CJ\otimes\CJ)\,\G^*\,  (1\otimes\CJ)\, \hat W.
$$
In this formula  $\CJ$ is the unitary operator on $L^2(G)$ associated to the group inversion, $\hat W$ is the multiplicative unitary of the dual
(quantum) group $\hat G=(W^*(G),\Dhat)$ and
$\G\colon L^2(G)\otimes L^2(G)\to L^2(G)\otimes L^2(G)$ is the unitary Galois map of the Galois object $(B(\CH),\Ad\pi)$. Explicitly, with $\Delta$ the modular function of $G$ and with $D$   the Duflo--Moore operator  of the projective representation~$\pi$,  it is given  by
$$
\big(\G(f_1\otimes f_2)\big)(g,h)=\Delta(g)^{-1/2}\Op^*\big(\Op(\lambda_{g}f_1)D^{-1/2}\Op(f_2)\big)(h).
$$

Note that when $\pi$ is a genuine representation, it is not difficult to show that  a unitary equivariant quantization map always exists (see \cite{BGNT3}*{Theorem 2.13}). However, an explicit construction of such a quantization map remains a nontrivial task. In \cites{BGNT3, BGNT4, GM} we have constructed a variant of the so-called Kohn--Nirenberg
quantization satisfying the required properties for a class  of abelian extensions
$
0\to V\to G\to Q\to 1.
$

When the representation space $\CH$ is $L^2(X)$, for $X$ a locally compact space endowed  with
a Radon measure $dx$, a Kohn--Nirenberg type
quantization can be formally defined quite generally: for $f\in C_c(G)$, the operator $\Op(f)$ is initially defined as the sesquilinear  form on $C_c(X)$ given by the formula
\begin{align}
\label{OPKN}
\Op(f)[\vf_1,\vf_2]:=\int_G f(g)\,\overline{(\pi(g)^*\vf_1)}(x_0)  \,\bigg(\int_X (\pi(g)^*\vf_2)(x)\,\mu(x)dx\bigg)dg,
\end{align}
where $x_0\in X$ is a fixed base point and   $\mu$ is a density.
For $G=\R^{2n}$ and for $\pi$  the projective representation on $L^2(\R^n)$ given by the restriction to $\R^{2n}$ of the Schrödinger representation of the
Heisenberg group $H_n$, this formula (for $x_0=0$ and $\mu=1$) reproduces  exactly the classical Kohn--Nirenberg quantization.
We do not claim that this formula always extends to a unitary $\Op\colon L^2(G)\to \HS(L^2(X))$, but at least the equivariance property
$\pi(g)\Op(f)\pi(g)^*=\Op(\lambda_gf)$ is automatic. It should be seen as an ansatz for a unitary equivariant quantization.

\smallskip

Using this ansatz, we construct here a unitary quantization map for semidirect products $G=H\ltimes V$ satisfying the \emph{dual orbit condition of depth $\ell$} (see Definition \ref{DOCn}).
A paradigmatic  example  in this class is the full affine group ${\rm Aff}(V)=\GL(V)\ltimes V$ of a finite dimensional vector space $V$
over a local skew field $\K$ (Archimedean or not), which already exhibits all the analytical difficulties involved in the general scheme.
In this example the representation theory is entirely described by the Mackey method. In particular, if we take any  point  $\xi_0$ in the main dual orbit $\mathcal O=\hat V\setminus\{0\}$, then it stabilizer is isomorphic to ${\rm Aff}(V')$,
where ${\rm dim}(V')={\rm dim}(V)-1$, and one concludes that  ${\rm Aff}(V)$ possesses a unique class of
square-integrable irreducible representations. A representative of this class  is inductively given  by the induced representation
$$
\pi:={\Ind}_{{\rm Aff}(V')\ltimes V}^{{\rm Aff}(V)}(\pi'\otimes \xi_{0}).
$$
However, with this choice of a representative it is difficult to give a precise meaning to~\eqref{OPKN}. One of the main results of this paper is another construction of $\pi$ that is much better suited for this task.

The crucial observation is that
any group $G$ satisfying the dual orbit condition of depth $\ell$ admits a double crossed product presentation $G=P\bowtie N$ and
 the closed subgroup $N$ always carries a nontrivial unitary character $\chi$. It turns out  that the Mackey representation is unitarily equivalent to
 ${\Ind}_{N}^G(\chi)$ and that the Kohn--Nirenberg quantization \eqref{OPKN} for this choice of representative is intimately related to a unitary
 scalar Fourier transform $\CF:L^2(N)\to L^2(P)$, which intertwines the left regular representations  of $P$ and $N$ with representations defined  by the dressing transformations.

In the case of the affine group $\Aff(V)$, the group $P$ is isomorphic to the parabolic group of  triangular matrices of size ${\rm dim}(V)$, and $N$
is isomorphic to the nilpotent group of  unitriangular matrices of size ${\rm dim}(V)+1$. This decomposition already appears in \cite{Medina}
for the connected affine group over the reals.

\section{Setup}

\subsection{Notation}
Let $G$ be a locally compact group, always assumed to be second countable. We fix a left-invariant Haar measure $dg$ on $G$.
The modular function $\Delta_G$ is defined by the relation
$$
\int_G f(gh)\,dg=\Delta_G(h)^{-1}\int_G f(g)\,dg\ \ \text{for}\ \ f\in C_c(G).
$$
In a similar way, for a continuous automorphism $\tau\in\Aut(G)$, its modulus $|\tau|_G$  is defined by the identity
$$
\int_Gf(\tau(g))\,dg=|\tau|_G^{-1}\int_Gf(g)\,dg\ \ \text{for}\ \ f\in C_c(G).
$$
When the automorphism comes from the conjugation $${\bf C}_x(g):=xgx^{-1}$$ by an element $x\in L$ of a group $L$
containing $G$ as a closed normal subgroup,  we use the shorthand notation $|x|_G$ for $|{\bf C}_x|_G$.
The multiplicative unitary $W_G\colon L^2(G\times G)\to L^2(G\times G)$ of~$G$ is defined by
$$
(W_Gf)(g,h):=f(g,g^{-1}h),
$$
and $\lambda$ and $\rho$  denote the left and right regular  representations of $G$ on $L^2(G)$:
\begin{align*}
(\lambda_gf)(h)=f(g^{-1}h)\quad\mbox{and}\quad(\rho_gf)(h)=\Delta_G(g)^{1/2}f(hg).
\end{align*}

Let $G_1,G_2$ be two locally compact groups. If we are given a continuous homomorphism $\mu:G_1\to \Aut(G_2)$,
we can consider the semidirect product $G= G_1\ltimes G_2$, so as a set $G=G_1\times G_2$ with the group law
$(g_1,g_2)(g'_1,g'_2)=(g_1\,g'_1, g_2\,\mu_{g_1}(g_2'))$.
When convenient,  we shall regard $G_1$ and $G_2$ as closed subgroups $G$ in the standard way. Since
$(e,g_2)(g_1,e)=(g_1,g_2)$,  it is natural to parameterize elements of $G$ as $g=g_2g_1$.
In this parametrization
the extension homomorphism is given by conjugation $\mu_{g_1}(g_2)={\bf C}_{g_1}(g_2)$, while the left-invariant Haar measure
and the modular function are given by
$$
dg=\frac{dg_1dg_2}{|g_1|_{G_2}},\quad \Delta_G(g)=\frac{\Delta_{G_1}(g_1)\,\Delta_{G_2}(g_2)}{|g_1|_{G_2}}.
$$

Let $G$ be a locally compact group and  let $G_1,G_2$ be two closed subgroups of $G$.
Recall that $(G_1,G_2)$ forms a matched pair for $G$ if
 $G_1\cap G_2=\{e\}$ and  $G_1 G_2$ is a subset of full measure in~$G$. We then say that $G$ is the double crossed product of $G_1$
 and $G_2$ and we write $G=G_1\bowtie G_2$.
In this situation there exist  measurable actions
\begin{equation*}
\alpha :G_1\times G_2\to G_2\quad \mbox{and}\quad
\beta:G_2\times G_1\to G_1,
\end{equation*}
such that for almost all $g_1\in G_1$ and $g_2\in G_2$ we have the relation:
\begin{align*}
g_1g_2^{-1}=\alpha _{g_1}(g_2)^{-1}\,\beta_{g_2}(g_1).
\end{align*}
The actions $\alpha$ and $\beta$ are not  by group automorphisms, but we have nevertheless control
on the images of the products (see \cite{VV}*{Lemma 4.9}): for  $g_1,\tilde g_1\in G_1$ and $g_2,\tilde g_2\in G_2$, we have
\begin{align}
\label{CP}
\alpha_{g_1 }(g_2\tilde g_2)=\alpha_{\beta_{\tilde g_2}(g_1 )}(g_2 )\,\alpha_{g_1 }(\tilde g_2 ),\quad \beta_{g_2 }(g_1\tilde g_1)=\beta_{\alpha_{\tilde g_1}(g_2 )}(g_1 )\,
\beta_{g_2}(\tilde g_1).
\end{align}

Let now  $V$ be a locally compact Abelian group and let $\hat V$ be its Pontryagin dual.
 We will use the additive notation both on $V$ and  on $\hat V$.
We denote the duality pairing by
$$
\hat V\times V\to \T,\quad (v,\xi)\mapsto e^{i\langle\xi,v\rangle}.
$$
 This is just a notation, we do not claim that
there is an exponential function here. To be consistent,  we also use the notation
$e^{-i\langle\xi,v\rangle}:=\overline{e^{i\langle\xi,v\rangle}}=e^{i\langle-\xi,v\rangle}=e^{i\langle\xi,-v\rangle}$.
Once  a Haar  measure~$dv$ has been fixed on $V$, we normalize the Haar measure $d\xi$ on $\hat V$ so that the Fourier
transform $\CF_V$ defined by
$$
(\CF_Vf)(\xi):=\int_V e^{-i\langle\xi,v\rangle} f(v)\,dv\quad\text{for}\quad f\in L^1(V)\cap L^2(V)
$$
extends to a unitary operator from $L^2(V)$ to $L^2(\hat V)$.

Given an action $G\times V\to V$ by group automorphisms, $(g,v)\mapsto g.v$, we denote by $G\times \hat V\to \hat V$, $(g,\xi)\mapsto g^\flat\xi$, the dual action, which is
defined by the identity  $e^{i\langle g^\flat\xi,v\rangle}=e^{i\langle\xi,g^{-1}.v\rangle}$. We then have $|g|_V=|g^\flat|_{\hat V}^{-1}$.

\subsection{The class of groups} \label{ssec:class}

Let  $H$  and $V$  be  nontrivial second countable locally compact  groups, with $V$ abelian.
We assume that we are given a continuous homomorphism $H\to{\rm Aut}(V)$, so that we can form a semidirect product $G:=H\ltimes V$.

Every pair $(\mathcal O, [\pi'])$, where $\mathcal O\subset \hat V$ is an orbit for the dual action of $H$ and $\pi'$ is an  irreducible unitary representation of
 the stabilizer $G'\subset H$ of an element $\xi_0\in \mathcal O$, defines an irreducible unitary representation of $G$:
\begin{equation}
\label{MR}
{\Ind}_{G'\ltimes V}^{G}(\pi'\otimes \xi_0),
\end{equation}
where  $\pi'\otimes \xi_0$ is the representation of $G'\ltimes V$   on $\CH_{\pi'}$
given  by
$(\pi'\otimes \xi_0)(g',v)=e^{i\langle\xi_0,v\rangle} \pi'(g')$. Note that $G'$-invariance of $\xi_0$ assures that $\pi'\otimes\xi_0$ is indeed a representation. We call~\eqref{MR} a \emph{Mackey representation}.

It is known that if the action of $H$ on $V$ is regular in the sense of Mackey, meaning that there exists a Borel set in $\hat V$ that intersects each dual orbit at exactly one point (see e.g$.$ \cite{Folland}*{p.~196}), then the  unitary dual  of $G$ is fully described by the Mackey representations (\cite{Folland}*{Theorem~6.43}). While the action is going to be regular in our examples, we do not need this property for our analysis. More importantly for us, it is also known that a Mackey representation~\eqref{MR} is square-integrable if and only if $\pi'$ is square-integrable and~$\mathcal O$ has positive measure in $\hat V$ (see \cite{ACVL}*{Theorem~2}).

Our main motivating example is  the full affine group $G={\rm GL}_n(\K)\ltimes \K^n$ of a local skew field~$\K$ (Archimedean or not).
Fixing a nontrivial unitary character of $\K^n$ implementing the self-duality $\hat {\K}^n\simeq \K^n$ and
 choosing $\xi_0=(1,0,\cdots,0)$, we find out that the dual orbit is
 $ \mathcal O=\K^n\setminus\{0\}$ and that the stabilizer equals
  $$
  G'=\bigg\{\begin{pmatrix}1&0\\m&Z\end{pmatrix}: \ Z\in {\rm GL}_{n-1}(\K),\; m\in \K^{n-1}\bigg\}\simeq {\rm GL}_{n-1}(\K)\ltimes \K^{n-1}.
  $$
Since the stabilizer is trivial for $n=1$, induction shows that $G$ possesses a unique class of square-integrable irreducible unitary representations,
with a representative given by the Mackey representation \eqref{MR}.

An additional property of the affine group, which is important for the construction of our quantization,
is that $G$ possesses another closed subgroup $Q$ such that $(Q,G')$ forms   a matched pair
for $H={\rm GL}_n(\K)$ and, setting $G':=H'\ltimes V'$, that $H'={\rm GL}_{n-1}(\K)$ normalizes $Q$. Indeed, we can take
 $$
  Q:=\bigg\{\begin{pmatrix}a&x\\0&1_{n-1}\end{pmatrix}:\ a\in\K^*,\; x\in \K^{n-1}\bigg\}\simeq \K^*\ltimes \K^{n-1}.
  $$

We shall see that all these properties are also satisfied for many Lie groups whose Lie algebras are Frobenius seaweeds.
This motivates the following definition.

\begin{definition}
\label{DOCn}
We say that $(H,V)$  satisfies the {\bf dual orbit condition of depth $1$} (${\rm DOC}_1$) if
there exists an element $\xi_0\in\hat V$ such that the map
\begin{equation}
\label{Phi0}
\phi:H\to\hat V,\quad q\mapsto q^\flat\xi_0,
\end{equation}
is a measure class isomorphism.
We say that $(H,V)$  satisfies the {\bf dual orbit condition of depth} $\ell\geq 2$ (${\rm DOC}_\ell$) if
the following conditions are satisfied:
\begin{enumerate}
\item there exists an element $\xi_0\in\hat V$ whose $H$-orbit has full measure in $\hat V$ and its
 stabilizer  is of the form $G'=H'\ltimes V'$ with $V'$ abelian;
\item  there exists  another closed subgroup $Q$ of $H$ such that $(Q, G')$ forms a matched pair for~$H$
 and  such that  $H'$ normalizes $Q$;
\item the pair $(H',V')$ satisfies ${\rm DOC}_{\ell-1}$.
\end{enumerate}
  \end{definition}

\begin{remark}
A semidirect product  $G=H\ltimes V$ such that $(H,V)$ satisfies ${\rm DOC}_\ell$
has thus the form $G=(Q\bowtie G')\ltimes V$,
  where $G'=H'\ltimes V'$  and $(H',V')$ satisfies  ${\rm DOC}_{\ell-1}$.
  Of course, $(Q\ltimes V,G')$ forms a matched pair for $G$ too and therefore we also have $G=(Q\ltimes V)\bowtie G'$. Moreover,
the pair $(Q,V)$ satisfies ${\rm DOC}_1$.
  \end{remark}

\begin{lemma}\label{lem:typeI}
If $(H,V)$  satisfies the dual orbit condition of depth $\ell\ge1$, then $W^*(G)$ is a type~I factor. In particular,
$G$ has a unique up to equivalence square-integrable irreducible unitary representation.
\end{lemma}

\begin{proof}
By assumption we have an $H$-equivariant measure class isomorphism $H/G'\cong \hat V$. Hence we get the following standard isomorphisms
$$
W^*(G)\cong H\ltimes W^*(V)\cong H\ltimes L^\infty(H/G')\cong W^*(G')\bar\otimes B(L^2(H/G')),
$$
and the lemma follows by induction on $\ell$.
\end{proof}
\begin{remark}
As its proof shows, Lemma \ref{lem:typeI} remains valid even when condition~$(2)$ of Definition~\ref{DOCn} is dropped.
\end{remark}

In addition to the affine group of a local skew field ${\rm GL}_n(\K)\ltimes \K^n$, as examples of groups satisfying ${\rm DOC}_{\ell}$ for some $\ell\ge1$  we can consider
the matrix amplifications ${\rm GL}_{nk}\ltimes {\rm Mat}_{nk,k}(\K)$, with $k\in\mathbb N^*$ arbitrary.
Besides this, there are many examples of matrix groups whose  Lie algebras are Frobenius seaweeds \cite{DK}. In the list of examples given  below we closely follow the decomposition method of~\cite{Panyushev}.

\begin{example}
\label{E1}
Let $H:={\rm SL}_3(\R)\times  {\rm GL}_2(\R)$ acting on $V:={\rm Mat}_{3,2}(\R)$ by $(A,B).M:=AMB^{-1}$ and let $G:=H\ltimes V$. Identifying $\hat V$
with ${\rm Mat}_{2,3}(\R)$,   let $\xi_0:=\begin{pmatrix}0&1&0\\0&0&1\end{pmatrix}$. The stabilizer of $\xi_0$ for the dual action of $H$ on $\hat V$ is given by
$$
G'=\Bigg\{\bigg(\begin{pmatrix} {\rm det}B^{-1}&n\\0&B\end{pmatrix},B\bigg):\ B\in  {\rm GL}_2(\R), \; n\in{\rm Mat}_{1,2}(\R)\Bigg\}.
$$
Note that $G'$ is isomorphic to ${\rm GL}_2(\R)\ltimes\R^2$, but for the action given by $B.n:={\rm det}B^{-1} nB^{-1}$.
Consider now the closed subgroup of $H$ given by
$$
Q:=\Bigg\{\bigg(\begin{pmatrix} {\rm det}Z^{-1}&0\\m&Z\end{pmatrix}, 1\bigg):\ Z\in  {\rm GL}_2(\R), \; m\in{\rm Mat}_{2,1}(\R)\Bigg\}.
$$
Clearly, $(Q,G')$ forms a matched pair for $H$. Moreover, writing $G'=H'\ltimes V'$ (according to the decomposition ${\rm GL}_2(\R)\ltimes\R^2$),
we observe that $H'$ normalizes $Q$. Under the identification $\hat V'={\rm Mat}_{2,1}(\R)$,
let $\xi_0':=\begin{pmatrix}0\\1\end{pmatrix}$. Then, the stabilizer of $\xi_0'$ for the dual action of $H'$ is given by
$$
G''=\bigg\{\begin{pmatrix} a&b\\0&a^{-2}\end{pmatrix}:\ a\in\R^*,\;b\in\R \bigg\},
$$
so it is isomorphic to $\R^*\ltimes\R$, but for the action given by $a.b:=a^3 b$. Clearly, $G''$ satisfies  the condition ${\rm DOC_1}$
and one concludes that $G$ satisfies  the condition ${\rm DOC_3}$.
\end{example}

We now give other examples of subgroups of ${\rm GL}_n(\K)$ (with $\K$ any local skew field) satisfying the condition ${\rm DOC_\ell}$ for small values of $n$.

\begin{example}
\label{E2}
Consider the subgroup of ${\rm GL}_3(\K)$ given by
$$
G:=\begin{pmatrix}
*&*&0\\0&*&0\\
0&*&1\end{pmatrix}.
$$
Then  $G$
satisfies   ${\rm DOC_1}$. Indeed, we have $G\cong (\K^*)^2\ltimes \K^2$ for the action $(a,c).(b,d):=(ac^{-1} b,c^{-1} d)$.
\end{example}

\begin{example}
\label{E3}
Consider the subgroups of ${\rm GL}_4(\K)$ given by
$$
G_1:=\begin{pmatrix}
*&*&*&*\\
*&*&*&*\\
0&0&*&*\\
0&0&0&1\end{pmatrix},\;
G_2:=\begin{pmatrix}
*&*&*&*\\
*&*&*&*\\
0&0&1&*\\
0&0&0&*\end{pmatrix},\;
G_3:=\begin{pmatrix}
*&*&*&0\\
*&*&*&0\\
0&0&*&0\\
0&0&*&1
 \end{pmatrix},\;
 G_4:=\begin{pmatrix}
*&*&0&0\\
0&*&0&0\\
0&*&*&*\\
0&0&0&1
 \end{pmatrix}.
$$
Then  $G_1, G_2, G_3$  satisfy   ${\rm DOC_2}$
and $G_4$ satisfies   ${\rm DOC_1}$.
We only give a proof for $G_3$.

Consider the closed subgroups of $G_3$ given by
 $$
 H_3:= \begin{pmatrix}
*&*&0&0\\
*&*&0&0\\
0&0&*&0\\
0&0&0&1
 \end{pmatrix},
\; V_3:=\begin{pmatrix}
1&0&*&0\\
0&1&*&0\\
0&0&1&0\\
0&0&*&1
 \end{pmatrix},
\;
 G'_3:=\begin{pmatrix}
*&*&0&0\\
0&1&0&0\\
0&0&1&0\\
0&0&0&1
 \end{pmatrix},
\;
 Q_3:=\begin{pmatrix}
1&0&0&0\\
*&*&0&0\\
0&0&*&0\\
0&0&0&1
 \end{pmatrix}.
$$
We have a semi-direct product decomposition $G_3=H_3\ltimes V_3$. Identifying $\hat V_3$ with the transpose of $V_3$, we define
 $$
 \xi_0:=
\begin{pmatrix}
1&0&0&0\\
0&1&0&0\\
0&1&1&1\\
0&0&0&1
 \end{pmatrix}.
 $$
 One easily checks that   the stabilizer of $\xi_0$ is $G'_3=H'_3\ltimes V'_3\cong\K^*\ltimes\K$, that
$(Q_3,G'_3)$ forms a matched pair for $H_3$, and that $H'_3$ normalizes $Q_3$.
\end{example}

\begin{example}
\label{E5}
Consider the subgroups of ${\rm GL}_5(\K)$ given by
\begin{align*}
G_1:=\begin{pmatrix}
 *&*&*&*&*\\
 *&*&*&*&*\\
 0&0&*&*&*\\
 0&0&*&*&*\\
 0&0&0&0&1
 \end{pmatrix},\;
G_2:=\begin{pmatrix}
 *&*&*&*&0\\
 *&*&*&*&0\\
 *&*&*&*&0\\
 0&0&0&*&0\\
 0&0&0&*&1
 \end{pmatrix}&,\,
G_3:= \begin{pmatrix}
 *&*&*&*&0\\
 *&*&*&*&0\\
 0&0&*&*&0\\
 0&0&*&*&0\\
 0&0&*&*&1
 \end{pmatrix},\\
  G_4:=\begin{pmatrix}
 *&*&*&*&0\\
 *&*&*&*&0\\
 0&0&*&*&0\\
 0&0&0&*&0\\
 0&0&0&*&1
 \end{pmatrix}&,\;
G_5:= \begin{pmatrix}
 *&*&*&*&0\\
 *&*&*&*&0\\
 0&0&1&*&0\\
 0&0&0&*&0\\
 0&0&0&*&*
 \end{pmatrix}.
 \end{align*}
 Then $G_2$ satisfies  ${\rm DOC_3}$ and
 $G_1, G_3, G_4,G_5$  all satisfy    ${\rm DOC_2}$. We only give a proof for~$G_2$ and~$G_4$.

We first write   $G_2=H_2\ltimes V_2$, where
$$
 H_2:= \begin{pmatrix}
 *&*&*&0&0\\
 *&*&*&0&0\\
 *&*&*&0 &0\\
 0&0&0&*&0\\
 0&0&0&0&1
 \end{pmatrix}
\quad\mbox{and}\quad
 V_2:=\begin{pmatrix}
 1&0&0&*&0\\
 0&1&0&*&0\\
 0&0&1&*&0\\
 0&0&0&1&0\\
 0&0&0&*&1
 \end{pmatrix}.
 $$
 Then we consider the following closed subgroups of $G_2$:
 $$
 G'_2:= \begin{pmatrix}
 *&*&*&0&0\\
 *&*&*&0&0\\
 0&0&1&0&0\\
 0&0&0&1&0\\
 0&0&0&0&1
 \end{pmatrix}
\quad\mbox{and}\quad
 Q_2:= \begin{pmatrix}
 1&0&0&0&0\\
 0&1&0&0&0\\
 *&*&*&0&0\\
 0&0&0&*&0\\
 0&0&0&0&1
 \end{pmatrix}.
$$
We see that $G'_2=H'_2\ltimes V'_2\cong {\rm GL}_2(\K)\ltimes \K^2$ is the stabilizer of
 $$
 \xi_0:=
 \begin{pmatrix}
 1&0&0&0&0\\
 0&1&0&0&0\\
 0&0&1&0&0\\
 0&0&1&1&1\\
 0&0&0&0&1
 \end{pmatrix}\in \hat V_2.
 $$
Moreover, we see that $(Q_2,G'_2)$ forms a matched pair for $H_2$ and  that
 $H'_2$ normalizes $Q_2$.

 Next, consider the following closed subgroups of $G_4$:
 $$
 H_4:= \begin{pmatrix}
 *&*&0&0&0\\
 *&*&0&0&0\\
 0&0&*&* &0\\
 0&0&0&*&0\\
 0&0&0&*&1
 \end{pmatrix}
,\;
 V_4:=\begin{pmatrix}
 1&0&*&*&0\\
 0&1&*&*&0\\
 0&0&1&0&0\\
 0&0&0&1&0\\
 0&0&0&0&1
 \end{pmatrix}
,\;
Q_4:= \begin{pmatrix}
 *&*&0&0&0\\
 *&*&0&0&0\\
 0&0&1&0&0\\
 0&0&0&1&0\\
 0&0&0&0&1
 \end{pmatrix}.
 $$
 Evidently, we have
 $G_4=H_4\ltimes V_4$. Identifying $\hat V_4$ with the transpose of $V_4$, we define
 $$
 \xi_0:=
 \begin{pmatrix}
 1&0&0&0&0\\
 0&1&0&0&0\\
 1&0&1&0&0\\
 0&1&0&1&0\\
 0&0&0&0&1
 \end{pmatrix}.
 $$
One then checks that $G'_4$,  the stabilizer  of $\xi_0$,  consists of all matrices of the form
$$
 \begin{pmatrix}
 a&b&0&0&0\\
 0&c&0&0&0\\
 0&0&a&b&0\\
 0&0&0&c&0\\
 0&0&0&d&1
 \end{pmatrix},
$$
so it is isomorphic to the group given  in Example \ref{E2}.
This finishes the proof, because we have $H_4=G'_4\ltimes Q_4$.
\end{example}

\subsection{A double crossed product presentation}
We need now to put some more effort into notation in order to keep track of different subgroups  appearing in the inductive definition of
a semidirect product $G_\ell=H_\ell\ltimes V_\ell$  satisfying  ${\rm DOC}_\ell$.

By definition, there exist $G_j,H_{j},V_j,Q_j$, $j=\ell,\cdots,1$, all closed subgroups of $G_\ell$, such that~$V_j$ is abelian,
$G_j=H_{j}\ltimes V_j$ and  $H_j=Q_j\bowtie G_{j-1}$ (with $H_1=Q_1$). We also let $\xi_{0,j}\in\hat V_{j}$ be the  element in the main  $H_{j}$-orbit in $\hat V_j$
such that $G_{j-1}={\rm Stab}_{H_{j}}(\xi_{0,j})$.

Note also that  the pairs $(Q_j,V_j)$, $j=\ell,\cdots,1$, all
satisfy  ${\rm DOC}_1$, and
therefore we can write the group $G_\ell$
as an iterated double crossed product of semidirect products all satisfying ${\rm DOC}_1$:
\begin{equation}
\label{IDCP}
G_\ell=(Q_\ell\ltimes V_\ell)\bowtie\big((Q_{\ell-1}\ltimes V_{\ell-1})\bowtie\big(\cdots\bowtie(Q_1\ltimes V_1)\cdots\big)\big).
\end{equation}
However, this description has  some  technical drawbacks, so instead we are going
to write  $G_\ell$ as a double crossed product involving a  single matched pair.

Let $P_\ell$ and $N_\ell$ be the subgroups of $G_\ell$ generated respectively by the  subgroups $Q_\ell, \cdots, Q_1$ and
 by  $V_\ell, \cdots, V_1$.
By definition,  we have $V_i\subset H_j$ for $i<j$. Hence $V_i$ normalizes $V_j$ and therefore~$N_\ell$
is an iterated semidirect product of abelian factors:
$$
N_\ell=N_{\ell-1}\ltimes V_\ell=((\cdots(V_1\ltimes V_2)\ltimes\cdots)\ltimes V_{\ell-1})\ltimes V_\ell.
$$
Similarly, we have $Q_i\subset H_i\subset H_j$, for $i<j$. Hence $Q_i$ normalizes $Q_j$, so $P_\ell$
is also an iterated semidirect product:
$$
P_\ell=P_{\ell-1}\ltimes Q_\ell=((\cdots(Q_1\ltimes Q_2)\ltimes\cdots)\ltimes Q_{\ell-1})\ltimes Q_\ell.
$$
\begin{example}
For the affine group $G_\ell={\rm GL}_\ell(\K)\ltimes \K^\ell$, $P_\ell$ is the parabolic subgroup of  upper triangular matrices in ${\rm GL}_\ell(\K)$  and
 $N_\ell$   is the semidirect product of the nilpotent subgroup of lower unitriangular matrices in ${\rm GL}_\ell(\K)$ acting on $\K^\ell$.
 Hence $N_\ell$ is  isomorphic to the  group of lower unitriangular matrices in ${\rm GL}_{\ell+1}(\K)$.
  \end{example}

\begin{proposition}
We have $G_\ell=P_\ell\bowtie N_\ell$.
\end{proposition}
\begin{proof}
We have to show that  $(P_\ell,N_\ell)$ forms a matched pair for $G_\ell$.
That $P_\ell\cap N_\ell=\{e\}$ is obvious. Next, we see by \eqref{IDCP} that almost every  $g_\ell\in G_\ell$ can be written as a product
\begin{equation}
\label{P1}
g_\ell=v_\ell q_\ell v_{\ell-1} q_{\ell-1}\cdots v_1q_1,\quad \mbox{where}\quad q_j\in Q_j,\, v_j\in V_j.
\end{equation}
Note that $Q_j\subset H_i$ for $j\leq i$. Hence $Q_j$ normalizes $V_i$. Passing the $q$'s through the $v$'s on the left, we see
that  almost all $g_\ell\in G_\ell$ can be written in the form
$$
g_\ell=\tilde q_\ell \tilde q_{\ell-1}\cdots \tilde q_1 \tilde v_\ell \tilde v_{\ell-1}\cdots \tilde  v_1 ,\quad \mbox{where}\quad \tilde  q_j\in Q_j,\, \tilde  v_j\in V_j .
$$
Therefore $P_\ell N_\ell$ has full measure in $G_\ell$.
\end{proof}

Consider the associated  measurable actions $\alpha \colon P_\ell\times N_\ell\to N_\ell$ and
$\beta:N_\ell\times P_\ell\to P_\ell$, such that
\begin{align}
\label{MA}
p_\ell n_\ell^{-1}=\alpha _{p_\ell}(n_\ell)^{-1}\,\beta_{n_\ell}(p_\ell).
\end{align}
In principle we should have used more precise notation. Namely, we should have written
$\alpha_j \colon P_j\times N_j\to N_j$ and
$\beta_j\colon  N_j\times P_j\to P_j$, $j=\ell,\cdots,1$  to distinguish these actions  at different depths.
But this is unnecessary  here,  since $\alpha_\ell|_{P_j\times N_j}=\alpha_j$ and $\beta_\ell|_{N_j\times P_j}=\beta_j$.

\begin{lemma}
\label{CaseS}
The restriction $\beta\big|_{V_\ell} $ is trivial,
  $\alpha\big|_{P_\ell}$ preserves $N_{\ell-1}$ and $\beta\big|_{N_{\ell-1}}$ preserves $Q_\ell$.\\
 Moreover, for $q_i\in Q_i$ and $v_j\in V_j$ with $i,j=\ell,\cdots,1$, we have:
$$
\alpha_{q_i}(v_j)=\begin{cases}
{\bf C}_{q_i}(v_j),\quad i\leq j\\
v_j,\quad i\geq j+2
\end{cases},\qquad
\beta_{v_j}(q_i)=\begin{cases}
q_i,\quad i\leq j\\
{\bf C}_{v_j}(q_i),\quad i\geq j+2
\end{cases}.
$$
\end{lemma}
\begin{proof}
The first statement follows from the fact that $P_\ell$ acts on $V_\ell$ by conjugation.
The second statement follows from $H_\ell\cap N_\ell=N_{\ell-1}$ together with the equality
$$
N_\ell\ni \alpha_{p_\ell}(n_{\ell-1})=\beta_{n_{\ell-1}}(p_\ell)n_{\ell-1}p_\ell^{-1}\in H_\ell.
$$
The third statement follows from the fact that $(Q_\ell,G_{\ell-1})$ forms a matched pair for $H_\ell$ and that $N_{\ell-1}\subset G_{\ell-1}$.
The final statement follows from the fact that $Q_i\subset H_i$ normalizes $V_j$ for $i<j+1$ and $V_j\subset H_{i-1}$ normalizes $Q_j$
for $i>j+1$.
\end{proof}

\subsection{Haar measures and modular functions}
We will use the following notation. Given an element $n_\ell=v_\ell\cdots v_1\in N_\ell$, with $v_i\in V_i$, we denote by $n_{\ell-1}$ the element $v_{\ell-1}\cdots v_1$. Similarly, given an element $p_\ell=q_\ell\cdots q_1\in P_\ell$, we let $p_{\ell-1}:=q_{\ell-1}\cdots q_1$. Therefore
\begin{equation}
\label{NormPN}
n_\ell=v_\ell n_{\ell-1}=v_\ell\cdots v_1\in N_\ell \qquad\mbox{and}\qquad p_\ell=q_\ell p_{\ell-1}=q_\ell\cdots q_1\in P_\ell.
\end{equation}

In this notation, the  left-invariant  Haar measures and the modular functions
are inductively given by:
\begin{align}
\label{HMPN}
 dn_\ell=\frac{dv_\ell dn_{\ell-1}}{|n_{\ell-1}|_{V_\ell}},
\quad
dp_\ell=\frac{dq_\ell dp_{\ell-1}}{|p_{\ell-1}|_{Q_\ell}},
 \end{align}
\begin{align}
\label{MFPN}
\Delta_{N_\ell}(n_\ell)=\frac{\Delta_{N_{\ell-1}}(n_{\ell-1})}{|n_{\ell-1}|_{V_\ell}}, \quad
\Delta_{P_\ell}(p_\ell)=\frac{\Delta_{Q_\ell}(q_\ell)\Delta_{P_{\ell-1}}(p_{\ell-1}) }{ |p_{\ell-1}|_{Q_\ell}}.
\end{align}

It will also be convenient to consider the measurable actions $\alpha,\beta$ conjugated by the group inversion, that is, the actions $\tilde\alpha :P_\ell\times N_\ell
\to N_\ell$ and $\tilde\beta: N_\ell\times P_\ell\to P_\ell$ given by
$$
\tilde \alpha _{p_\ell}(n_\ell):=\big(\alpha _{p_\ell}(n_\ell^{-1})\big)^{-1},\quad
\tilde\beta_{n_\ell }(p_\ell ):=\big(\beta_{n_\ell }(p_\ell ^{-1})\big)^{-1}.
$$

 We start with a series of results, all based on measure-theoretical considerations, leading to important simplifications
of the formulas for the modular and modulus functions and the Haar measures.
\begin{lemma}
\label{Lphi}
The  measure class  isomorphism
 \begin{equation}
 \label{phi}
\phi_\ell:Q_\ell\to \hat V_\ell, \quad q_\ell\mapsto q_\ell^\flat\xi_{0,\ell},
\end{equation}
intertwines the dual action of $Q_\ell$  with the left action of $Q_\ell$,
it intertwines the dual action of $V_{\ell-1}$ with the action $\tilde\beta$
and it intertwines the dual action of $H_{\ell-1}$ with the conjugation action.
\end{lemma}
\bp
For $\tilde h_\ell=\tilde q_\ell\tilde g_{\ell-1}=\tilde q_\ell\tilde v_{\ell-1}\tilde h_{\ell-1}\in H_\ell$ and $q_\ell\in Q_\ell$, we have,
since $H_{\ell-1}$ normalizes $Q_\ell$, that
$$
\tilde h_\ell \,q_\ell=\tilde q_\ell \,\tilde\beta_{v_{\ell-1}}({\bf C}_{\tilde h_{\ell-1}}(q_\ell))\, \alpha_{{\bf C}_{\tilde h_{\ell-1}}(q_\ell^{-1})}(v_{\ell-1})\,
 \tilde h_{\ell-1}.
$$
Since $\xi_{0,\ell}$ is invariant under the dual action of $G_{\ell-1}$, we get
$$
{\tilde h_\ell} ^\flat\phi_\ell(q_\ell)=\phi_\ell\big(\tilde q_\ell \,\tilde\beta_{v_{\ell-1}}({\bf C}_{\tilde h_{\ell-1}}(q_\ell))\big),
$$
which is all we need.
\ep

\begin{lemma}
\label{IDMF}
We have $|\cdot|_{V_\ell}=|\cdot|_{Q_\ell}^{-1}$ on $P_{\ell-1}$.
\end{lemma}
\begin{proof}
Consider the action of $P_\ell$ on $\hat V_\ell\times P_{\ell-1}$ given  by
$\tilde p_\ell.(\xi_\ell, p_{\ell-1}):=(\tilde p_\ell^\flat \xi_\ell, \tilde p_{\ell-1} p_{\ell-1})$, and observe that the measure
$|\phi_\ell^{-1}(\xi_\ell)|_{\hat V_\ell}^{-1}|p_{\ell-1}|_{\hat V_\ell}^{-1} d\xi_\ell dp_{\ell-1}$ is invariant under this action.
Now,  since $P_{\ell-1}\subset H_{\ell-1}$, we deduce from Lemma \ref{Lphi} that the  measure class  isomorphism
$$
\Psi_\ell:P_\ell\to \hat V_\ell\times P_{\ell-1},\quad p_\ell=q_\ell p_{\ell-1}\mapsto (\phi_\ell(q_\ell),p_{\ell-1}),
$$
intertwines the left action of $P_\ell$ on itself with the action described above. Hence the pullback of the invariant measure on $\hat V_\ell\times P_{\ell-1}$   by
$\Psi_\ell$
is a multiple of the left-invariant Haar measure of $P_\ell$, which is $|p_{\ell-1}|_{Q_\ell}^{-1} dq_\ell dp_{\ell-1}$.
However, a direct computation shows  that this pullback  is  $|p_{\ell-1}|_{\hat V_\ell}^{-1} dq_\ell dp_{\ell-1}$. Therefore, there exists $c>0$ such that
for all $p_{\ell-1}\in P_{\ell-1}$, we have
$|p_{\ell-1}|_{\hat V_\ell}=c|p_{\ell-1}|_{Q_\ell}$. Evaluating this relation at the neutral element gives $c=1$.
We arrive at our conclusion  using the relation $|\cdot|_{\hat V_\ell}=|\cdot|_{V_\ell}^{-1}$.
\end{proof}

\begin{lemma}
\label{L1}
We have $\Delta_{G_\ell}\big|_{G_{\ell-1}}=\Delta_{G_{\ell-1}}$.
\end{lemma}
\bp
It is a classical result in harmonic analysis that the relation we have to prove is equivalent to existence of a $G_\ell$-invariant Radon measure
on the homogeneous space $G_\ell/G_{\ell-1}$. Consider the   measure class  isomorphism:
\begin{equation*}
\Theta_\ell:G_\ell/G_{\ell-1}\to \hat V_\ell\times V_\ell,\quad (v_\ell q_\ell g_{\ell-1}) G_{\ell-1}\mapsto (\phi_\ell(q_\ell),v_\ell).
\end{equation*}

Now, consider  the affine  action of $G_\ell$  on $ \hat V_\ell\times V_\ell$
defined by $(\tilde  v_\ell \tilde h_\ell ).(\xi_\ell, v_\ell):=(\tilde  h_\ell ^\flat\xi_\ell,\tilde v_\ell{\bf C}_{\tilde h_\ell}  ( v_\ell)  )$.
From Lemma \ref{Lphi} we see that $\Theta_\ell$ intertwines this action with  the one on the homogeneous space $G_\ell/G_{\ell-1}$:
$$
\tilde g_\ell.\Theta_\ell(g_\ell G_{\ell-1})=
\Theta_\ell\big(\tilde v_\ell{\bf C}_{\tilde h_\ell}  ( v_\ell)\tilde q_\ell\tilde\beta_{\tilde v_{\ell-1}}{\bf C}_{\tilde h_{\ell-1}}(q_\ell)G_{\ell-1}\big)
=\Theta_\ell(\tilde g_\ell g_\ell G_{\ell-1}).
$$
Since $|h_\ell|_{V_\ell}=|h_\ell|_{\hat V_\ell}^{-1}$,  the Haar measure $d\xi_\ell dx_\ell$ is invariant under the affine  action
of $G_\ell$ on $ \hat V_\ell\times V_\ell$, and therefore the pullback of this measure under $\Theta_\ell$ is the desired
invariant measure on  $G_\ell/G_{\ell-1}$.
\ep

  \begin{corollary}
  \label{C1L1}
For $p_\ell=q_\ell p_{\ell-1}\in P_\ell$ and $n_\ell\in N_\ell$,   we have
  $  \Delta_{G_\ell}(p_\ell)= \Delta_{G_\ell}(q_\ell)\Delta_{G_{\ell-1}}(p_{\ell-1})$,
  $\Delta_{G_\ell}(n_\ell)=1$ and $ \Delta_{G_\ell}\big(\tilde\beta_{n_\ell}(p_\ell)\big)= \Delta_{G_\ell}(p_\ell)$.
   \end{corollary}
  \begin{proof}
 Since $P_{\ell-1}\subset G_{\ell-1}$, we get
 $ \Delta_{G_\ell}(q_\ell p_{\ell-1})= \Delta_{G_\ell}(q_\ell)\Delta_{G_\ell}(p_{\ell-1})=\Delta_{G_\ell}(q_\ell)\Delta_{G_{\ell-1}}(p_{\ell-1})$,
 which gives the first relation.
Since $G_\ell=H_\ell\ltimes V_\ell$   we have $\Delta_{G_\ell}(v_\ell)=1$ and since $N_{\ell-1}\subset G_{\ell-1}$, we get
$ \Delta_{G_\ell}(v_\ell n_{\ell-1})
 =\Delta_{G_\ell}(n_{\ell-1})= \Delta_{G_{\ell-1}}(n_{\ell-1})$, from which the second relation follows.
 The last relation follows then from  \eqref{MA}.
 \end{proof}

  \begin{corollary}
  \label{C2L1}
 The left-invariant Haar measure on  $Q_\ell$ is invariant under the action $\tilde\beta$ of $N_{\ell-1}$.
 \end{corollary}
  \begin{proof}
   Consider  the matched pair $(Q_\ell,G_{\ell-1})$ for $H_\ell$. Since $N_{\ell-1}\subset G_{\ell-1}$, we know from  \cite{VV}*{Lemma 4.12}  that
 for any positive Borel function $f_1\colon Q_\ell\to\R_+$ the following holds:
$$
\int_{Q_\ell} f_1\big(\beta_{n_{\ell-1}}(q_\ell)\big) dq_\ell=
\int_{Q_\ell} f_1(q_\ell)\frac{\Delta_{G_\ell}\big(\tilde\alpha_{q_\ell}(n_{\ell-1})\big)} {\Delta_{G_{\ell-1}}\big(\tilde\alpha_{q_\ell}(n_{\ell-1})\big)}
\frac{\Delta_{Q_\ell}\big(\beta_{n_{\ell-1}^{-1}}(q_\ell)\big)}{\Delta_{Q_\ell}(q_\ell)}dq_\ell.
$$
Because $\Delta_{G_\ell} \big|_{N_\ell}=1$ by Corollary \ref{C1L1}, we deduce that
$$
\int_{Q_\ell} f_1\big(\beta_{n_{\ell-1}}(q_\ell)\big) dq_\ell=\int_{Q_\ell} f_1(q_\ell)
\frac{\Delta_{Q_\ell}\big(\beta_{n_{\ell-1}^{-1}}(q_\ell)\big)}{\Delta_{Q_\ell}(q_\ell)}dq_\ell.
$$
Expressing this equality in terms of the function $f_2(q_\ell)= f_1(q_\ell^{-1})$ and
performing the change of variable $q_\ell\mapsto q_\ell^{-1}$, we get:
$$
\int_{Q_\ell} f_2\big(\tilde\beta_{n_{\ell-1}}(q_\ell)\big) \Delta_{Q_\ell}(q_\ell)^{-1}dq_\ell=\int_{Q_\ell} f_2(q_\ell)
\Delta_{Q_\ell}\big(\tilde \beta_{n_{\ell-1}^{-1}}(q_\ell)\big)^{-1}dq_\ell .
$$
Applying this to the function $f= \Delta_{Q_\ell}\big(\tilde \beta_{n_{\ell-1}^{-1}}(\cdot)\big)^{-1}\, f_2$,
we get
$$
\int_{Q_\ell}f\big(\tilde\beta_{n_{\ell-1}}(q_\ell)\big)\,dq_\ell=\int_{Q_\ell} f(q_\ell)\,dq_\ell,
$$
which concludes the proof.
 \end{proof}

 \begin{proposition}
\label{NUM}
The group $N_\ell$ is unimodular.
\end{proposition}
\begin{proof}
We have seen in Lemma \ref{Lphi} that the  measure class  isomorphism $\phi_\ell:Q_\ell\to \hat V_\ell$ defined in \eqref{phi}
intertwines the action $\tilde\beta$ of $N_{\ell-1}$ on $Q_\ell$ with
the dual action of $N_{\ell-1}$ on $\hat V_\ell$.
Since the pull-back of the Haar  measure of $\hat V_\ell$ under the map $\phi_\ell$ is $|q_\ell|_{V_\ell}^{-1}dq_\ell$,
for any Borel function $f_1:Q_\ell\to\R_+$, we get
$$
|n_{\ell-1}|_{V_\ell}^{-1}\int_{Q_\ell}f_1\big(\tilde\beta_{n_{\ell-1}}(q_\ell)\big)\,|q_\ell|_{V_\ell}^{-1}dq_\ell=\int_{Q_\ell} f_1(q_\ell)\,|q_\ell|_{V_\ell}^{-1}dq_\ell.
$$
In terms of the function $f=|\cdot|_{V_\ell}^{-1}\,f_1$, this means that
$$
\int_{Q_\ell}f\big(\tilde\beta_{n_{\ell-1}}(q_\ell)\big)\,|n_{\ell-1}|_{V_\ell}^{-1}\,|q_\ell|_{V_\ell}^{-1}\,\big|\tilde\beta_{n_{\ell-1}}(q_\ell)\big|_{V_\ell}dq_\ell=\int_{Q_\ell} f(q_\ell)\,dq_\ell.
$$
Since $n_{\ell-1}\, q_\ell=\tilde\beta_{n_{\ell-1}}(q_\ell)\,\alpha_{q_\ell^{-1}}(n_{\ell-1})$,
we finally obtain
$$
\int_{Q_\ell}f\big(\tilde\beta_{n_{\ell-1}}(q_\ell)\big)\,\big|\alpha_{q_\ell^{-1}}(n_{\ell-1})\big|_{V_\ell}^{-1}dq_\ell=\int_{Q_\ell} f(q_\ell)\,dq_\ell.
$$
However, we have seen in Corollary \ref{C2L1} that  the Haar measure on  $Q_\ell$ is invariant under  $\tilde\beta$.
Therefore we get $(|\cdot|_{V_\ell}\big)|_{N_{\ell-1}}=1$, and the result follows from the expression \eqref{MFPN} for the modular function of $N_\ell$.
\end{proof}

\begin{remark}
\label{InvM}
From the proof of Proposition \ref{NUM} and the relation \eqref{MA}, we see that the restriction  to $P_\ell$ of the modulus function of $V_\ell$ is $\tilde\beta$-invariant:
$$
\big|\tilde\beta_{n_\ell}(p_\ell)\big|_{V_\ell}=|p_\ell|_{V_\ell}.
$$
\end{remark}

 \begin{remark}
 \label{IF1}
Using \cite{VV}*{Lemma 4.12} exactly like we did in the proof of Corollary \ref{C2L1}, the unimodularity of $N_\ell$  allows us to prove that
  for   $f\in L^1(N_\ell)$ and $p_\ell\in P_\ell$, we have:
$$
\int_{N_\ell}f\big(\tilde\alpha_{p_\ell}(n_\ell)\big)\, J_{\tilde\alpha_\ell}(p_\ell,n_\ell)\,dn_\ell
=\int_{N_\ell} f(n_\ell)\,dn_\ell,
$$
where
$$
J_{\tilde\alpha_\ell}(p_\ell,n_\ell)=\Delta_{G_\ell}^{-1}(p_\ell) \,\Delta_{P_\ell}\big(\beta_{n_{\ell-1}^{-1}}(p_\ell)\big) .
$$
Since moreover $\tilde\alpha_{q_\ell}(v_\ell \,n_{\ell-1})={\bf C}_{q_\ell}(v_\ell) \,\tilde\alpha_{q_\ell}(n_{\ell-1})$, we get from
\eqref{HMPN} that
$$
\int_{N_{\ell-1}}f\big(\tilde\alpha_{q_\ell}(n_{\ell-1})\big)\, |q_\ell|_{V_\ell}^{-1}\,J_{\tilde\alpha_\ell}(q_\ell,n_{\ell-1})\,dn_{\ell-1}
=\int_{N_{\ell-1}} f(n_{\ell-1})\,dn_{\ell-1},
$$
for   $f\in L^1(N_{\ell-1})$ and $q_\ell\in Q_\ell$.
 \end{remark}

\begin{corollary}
\label{C1P1}
The left-invariant Haar measure on $P_\ell$ is invariant under the action $\tilde\beta$ of $N_\ell$.
\end{corollary}
\begin{proof}
By Corollary  \ref{C1L1} and Proposition \ref{NUM}, we have $\Delta_{G_\ell}\big|_{N_\ell}=1 = \Delta_{N_\ell}$. Hence, the homogeneous space $G_\ell/N_\ell$ carries
a $G_\ell$-invariant measure. Consider the measure class isomorphism $P_\ell\to G_\ell/N_\ell$, $p_\ell\mapsto p_\ell N_\ell$. Trivially, this map
intertwines the restriction to $P_\ell$ of the action of $G_\ell$ on $G_\ell/N_\ell$, with the left action of $P_\ell$ on itself.
Therefore the pullback to  $P_\ell$ of the $G_\ell$-invariant measure on $G_\ell/N_\ell$ is a multiple of the left-invariant Haar measure.
But this map also
intertwines the restriction to $N_\ell$ of the  action of $G_\ell$ on $G_\ell/N_\ell$ with the action $\tilde\beta$ of $N_\ell$ on $P_\ell$.
Hence the left-invariant Haar measure of $P_\ell$ is $\tilde\beta$-invariant.
\end{proof}

\begin{corollary}
\label{C2P1}
The left-invariant Haar measures of $G_\ell$, $P_\ell$, and $N_\ell$ can be normalized such that for all $f\in L^1(G_\ell)$ we have
$$
\int_{G_\ell} f(g_\ell)\,dg_\ell=\int_{P_\ell\times N_\ell} f( p_\ell n_\ell)\,d p_\ell dn_\ell,
$$
\end{corollary}
\begin{proof}
We know  by  \cite{VV}*{Lemma 4.10} that the left-invariant Haar measures of $G_\ell$, $P_\ell$, and $N_\ell$ can be normalized such that for $ f\in L^1(G_\ell)$, we have
$$
\int_{G_\ell} f(g_\ell)\,dg_\ell=\int_{P_\ell\times N_\ell} f( p_\ell n_\ell^{-1})\,\Delta_{G_\ell}(n_\ell)^{-1}\,d p_\ell dn_\ell.
$$
The result follows because $\Delta_{G_\ell}\big|_{N_\ell}=1$  by Corollary  \ref{C1L1}  and because $N_\ell$ is unimodular by Proposition \ref{NUM}.
\end{proof}

 \section{Kohn--Nirenberg quantization}

\subsection{A scalar Fourier transform}
We start by observing that the unimodular group $N_\ell$ has a distinguished character.
\begin{lemma}
\label{Dchi}
The map $\chi_\ell :N_\ell\to \T$ given by
$$
\chi_\ell(v_\ell\cdots v_1):=e^{i\langle \xi_{0,\ell},v_\ell\rangle}\cdots e^{i\langle \xi_{0,1},v_1\rangle},
$$
defines a unitary character.
\end{lemma}
\begin{proof}
We proceed by induction. For $\ell=1$ there is nothing to prove.
So assume that $\chi_{\ell-1}$ is a character of $N_{\ell-1}$.
Take $n_\ell=v_\ell n_{\ell-1},\tilde n_\ell=\tilde v_\ell\tilde  n_{\ell-1}\in N_{\ell}$, with
$v_\ell,\tilde v_\ell\in V_\ell$ and  $n_{\ell-1},\tilde n_{\ell-1}\in N_{\ell-1}$. Since $N_{\ell-1}$ normalizes $V_\ell$,
we have $n_\ell \tilde n_\ell=v_\ell{\bf C}_{n_{\ell-1}}(\tilde  v_\ell)n_{\ell-1}\tilde  n_{\ell-1}$. From the relation
$\chi_\ell(v_\ell n_{\ell-1})=e^{i\langle \xi_{0,\ell},v_\ell\rangle}\chi_{\ell-1}(n_{\ell-1})$,
we deduce
$$
\chi_\ell(n_\ell \tilde  n_\ell)=e^{i\langle \xi_{0,\ell},v_\ell\rangle}e^{i\langle \xi_{0,\ell},{\bf C}_{n_{\ell-1}}(\tilde v_\ell)\rangle}\chi_{\ell-1}(n_{\ell-1}\tilde  n_{\ell-1}),
$$
and the proof follows, because $N_{\ell-1}\subset G_{\ell-1}$ acts trivially on $\xi_{0,\ell}\in\hat V_\ell$.
\end{proof}

An important function on $G_\ell$ is the following (almost  everywhere defined) Fourier type kernel:
\begin{align}
\label{FTK}
\E_\ell(p_\ell ,n_\ell):=\chi_\ell\big(\tilde\alpha_{ p_\ell^{-1}}(n_\ell)\big).
\end{align}
This function satisfies some nice identities.
\begin{lemma}
\label{SE}
We have  almost everywhere:
$$
\E_\ell(\tilde p_\ell^{-1}p_\ell,n_\ell)=\E_\ell\big(p_\ell,\tilde\alpha_{\tilde p_\ell}(n_\ell)\big)\quad\mbox{and}\quad
\E_\ell(p_\ell,\tilde n_\ell^{-1}n_\ell)=\E_\ell(p_\ell,\tilde n_\ell^{-1})\,\E_\ell\big(\tilde\beta_{\tilde n_\ell}(p_\ell),n_\ell\big).
$$
\end{lemma}
\begin{proof}
The first identity is obvious, and the second is a  consequence of the relation  \eqref{CP}:
\begin{multline*}
\chi_\ell\big(\tilde\alpha_{ p_\ell^{-1}}(\tilde n_\ell^{-1}n_\ell)\big)=\overline{\chi}_\ell\big(\alpha_{ p_\ell^{-1}}( n_\ell^{-1}\tilde n_\ell)\big)\\
=\overline{\chi}_\ell\big(\alpha_{ \beta_{\tilde n_\ell}(p_\ell^{-1})}( n_\ell^{-1})\big)\overline{\chi}_\ell\big(\alpha_{ p_\ell^{-1}}(\tilde n_\ell)\big)
=\chi_\ell\big(\tilde\alpha_{ \tilde\beta_{\tilde n_\ell}(p_\ell)^{-1}}( n_\ell)\big)\chi_\ell\big(\tilde\alpha_{ p_\ell^{-1}}(\tilde n_\ell^{-1})\big).
\end{multline*}
\end{proof}
\begin{corollary}
We have
$\E_\ell\big(\tilde\beta_{n_\ell}(p_\ell),n_\ell\big)=\overline{\E_\ell}(p_\ell,n_\ell^{-1})$.
\end{corollary}

\begin{lemma}
\label{RecE}
We have the following inductive relation:
$$
\E_\ell(p_\ell ,n_\ell)=e^{i\langle \phi_\ell(q_\ell),v_\ell\rangle}\,\E_{\ell-1}\big(p_{\ell-1} ,\tilde\alpha_{q_\ell^{-1}}(n_{\ell-1})\big),
$$
where the map $\phi_\ell:Q_\ell\to\hat V_\ell$ is given by \eqref{phi}.
\end{lemma}
\begin{proof}
By \eqref{CP} we have:
$$
\tilde\alpha_{ p_\ell^{-1}}(n_\ell)=\alpha_{ p_\ell^{-1}}(n_{\ell-1}^{-1}v_\ell^{-1})^{-1}=
\big(\alpha_{ \beta_{v_\ell^{-1}}(p_\ell^{-1})}(n_{\ell-1}^{-1})\alpha_{ p_\ell^{-1}}(v_\ell^{-1})\big)^{-1},
$$
which gives
$$
\tilde\alpha_{ p_\ell^{-1}}(n_\ell)=
\big(\alpha_{ p_\ell^{-1}}(n_{\ell-1}^{-1}){\bf C}_{ p_\ell^{-1}}(v_\ell^{-1})\big)^{-1}={\bf C}_{ p_\ell^{-1}}(v_\ell)\tilde\alpha_{ p_\ell^{-1}}(n_{\ell-1})
={\bf C}_{ p_\ell^{-1}}(v_\ell)\tilde\alpha_{ p_{\ell-1}^{-1}}(\tilde\alpha_{q_\ell^{-1}}(n_{\ell-1}))
$$
since $P_\ell$ acts by conjugation on $V_\ell$ (and $V_\ell$ acts trivially on $P_\ell$).
Therefore we get
$$
\chi_\ell\big(\tilde\alpha_{ p_\ell^{-1}}(n_\ell)\big)=\chi_\ell\big({\bf C}_{ p_\ell^{-1}}(v_\ell)\big)\chi_\ell\big(\tilde\alpha_{ p_{\ell-1}^{-1}}(\tilde\alpha_{q_\ell^{-1}}(n_{\ell-1}))\big)
=e^{i\langle p_\ell^\flat\xi_{0,\ell},v_\ell\rangle}\chi_{\ell-1}\big(\tilde\alpha_{ p_{\ell-1}^{-1}}(\tilde\alpha_{q_\ell^{-1}}(n_{\ell-1}))\big).
$$
This completes the proof, since  $p_{\ell-1}\in{\rm Stab}_{H_\ell}(\xi_{0,\ell})$, so $p_\ell^\flat\xi_{0,\ell}=q_\ell^{\flat}(p_{\ell-1}^\flat \xi_{0,\ell})=q_\ell^{\flat}\xi_{0,\ell}$.
\end{proof}

We will see soon that $\E_\ell$ is the phase of the operator kernel of a Fourier-type  transform from $L^2(N_\ell)$ to $L^2(P_\ell)$.
In order to define this transform, consider the unitary  operators $U_{\phi_\ell}:L^2(\hat V_\ell)\to L^2(Q_\ell)$ and $V_{\tilde\alpha_\ell}:L^2(Q_\ell\times N_{\ell-1})\to L^2(Q_\ell\times N_{\ell-1})$ defined by
$$
(U_{\phi_\ell}f)(q_\ell):=|q_\ell |_{V_\ell}^{-1/2}f(\phi_\ell(q_\ell)),
$$
$$
(V_{\tilde\alpha_\ell}f)(q_\ell,n_{\ell-1}):= |q_\ell|_{V_\ell}^{-1/2}\,
J_{\tilde\alpha_\ell}(q_\ell,n_{\ell-1})^{1/2}\,
f(q_\ell,\tilde\alpha_{q_\ell}(n_{\ell-1})),
$$
where $J_{\tilde\alpha_\ell}$ is the function defined in Remark~\ref{IF1}.
It will also be convenient to use the following standard unitary operators:
$$
V_{1,\ell}:L^2(P_\ell)\to L^2(Q_\ell\times P_{\ell-1}),\quad (V_{1,\ell}f)(q_\ell,p_{\ell-1}):= |p_{\ell-1}|_{Q_\ell}^{-1/2}\,f(q_\ell p_{\ell-1}),
$$
$$
V_{2,\ell}:L^2(N_\ell)\to L^2(V_\ell \times N_{\ell-1}),\quad (V_{2,\ell}f)(v_\ell,n_{\ell-1}):=f(v_\ell n_{\ell-1}).
$$

\begin{definition}
\label{UFT}
Let $\F_\ell:L^2(N_\ell)\to L^2(P_\ell)$ be the unitary operator defined inductively by $\F_1:=V_{\phi_1}\,\F_{V_1}$ and
$$
\F_\ell:=V_{1,\ell}^*\,(1\otimes \F_{\ell-1})\,V_{\tilde\alpha_\ell}\,( U_{\phi_\ell}\,\F_{V_\ell}\otimes 1)\,V_{2,\ell}.
$$
\end{definition}
\begin{remark}
\label{FFD}
At  the formal level it is not difficult  to see that the  operator $\F_\ell$ is an integral transform with kernel
$ \overline{\E_\ell}(p_\ell ,n_\ell)\, J_{\tilde\alpha_\ell}(p_{\ell}^{-1},n_\ell)^{1/2}$.
Indeed, consider the integral operator
$$
\tilde\F_\ell f(p_\ell):=\int_{N_\ell}\, \overline{\E_\ell}(p_\ell ,n_\ell)\, J_{\tilde\alpha_\ell}(p_{\ell}^{-1},n_\ell)^{1/2}\,f(n_\ell)\,dn_\ell.
$$
We have by Lemma \ref{RecE} that
$$
\E_\ell(p_\ell ,n_\ell)=e^{i\langle \phi_\ell(q_\ell),v_\ell\rangle}\,\E_{\ell-1}\big(p_{\ell-1} ,\tilde\alpha_{q_\ell^{-1}}(n_{\ell-1})\big).
$$
Moreover, the relation
\begin{align}
\label{TRL}
\tilde\alpha_{ p_\ell^{-1}}(n_\ell)
={\bf C}_{ p_\ell^{-1}}(v_\ell)\tilde\alpha_{ p_{\ell-1}^{-1}}(\tilde\alpha_{q_\ell^{-1}}(n_{\ell-1})),
\end{align}
which was used in the proof of Lemma \ref{RecE}, immediately gives
$$
J_{\tilde\alpha_\ell}(p_{\ell}^{-1},n_\ell)=|p_\ell|^{-1}_{V_\ell}\,|q_\ell|_{V_\ell}\,J_{\tilde\alpha_\ell}(q_\ell^{-1},n_{\ell-1})\,
J_{\tilde\alpha_{\ell-1}}\big(p_{\ell-1}^{-1},\tilde\alpha_{q_\ell^{-1}}(n_{\ell-1})\big).
$$
Using Lemma \ref{IDMF} to write $|p_\ell|_{V_\ell}=|q_\ell|_{V_\ell}|p_{\ell-1}|_{Q_\ell}^{-1}$, we therefore get
\begin{multline*}
\tilde\F_\ell f(p_\ell)=|p_{\ell-1}|_{Q_\ell}^{1/2}\int_{V_\ell\times N_{\ell-1}} e^{-i\langle \phi_\ell(q_\ell),v_\ell\rangle}\,\overline{\E_{\ell-1}}\big(p_{\ell-1} ,\tilde\alpha_{q_\ell^{-1}}(n_{\ell-1})\big)\\
J_{\tilde\alpha_\ell}(q_\ell^{-1},n_{\ell-1})^{1/2}\,
J_{\tilde\alpha_{\ell-1}}\big(p_{\ell-1}^{-1},\tilde\alpha_{q_\ell^{-1}}(n_{\ell-1})\big)^{1/2}\,f(v_\ell n_{\ell-1})\,dv_\ell dn_{\ell-1},
\end{multline*}
which by induction is seen to be exactly $\F_\ell$.
But unless we can show that the map $(p_\ell ,n_\ell)\mapsto J_{\tilde\alpha}(p_{\ell}^{-1},n_\ell)^{1/2}$ is locally integrable, all this remains formal
and therefore we will  keep the initial definition of $\F_\ell$.
\end{remark}

\begin{definition}
\label{URPN}
Let  $U_{\tilde \alpha_\ell}$ be the unitary representation of $P_\ell$   on $L^2(N_\ell)$ and let $U_{\tilde\beta_\ell}$ be the unitary representation of  $N_\ell$
 on $L^2(P_\ell)$ given by
$$
U_{\tilde \alpha_\ell}( p_\ell)f(n_\ell):=J_{\tilde\alpha_\ell}(p_\ell^{-1},n_\ell)^{1/2}\, f\big(\tilde\alpha_{p_\ell^{-1}}(n_\ell)\big)
\quad\mbox{and}\quad U_{\tilde \beta_\ell}( n_\ell)f(p_\ell):= f\big(\tilde\beta_{n_\ell^{-1}}(p_\ell)\big).
$$
\end{definition}

We have the following commutation relations.
\begin{proposition}
\label{CRF}
For $(p_\ell,n_\ell)\in P_\ell\times N_\ell$, we have
$$
\lambda_{  p_\ell}\,\F_\ell=\F_\ell\, U_{\tilde \alpha_\ell}(p_\ell) \quad\mbox{and}\quad \F_\ell\, \lambda_{n_\ell}=\overline\E_\ell(\cdot,n_\ell)\,U_{\tilde \beta_\ell}(n_\ell)\,\F_\ell.
$$
\end{proposition}
\begin{proof}
In addition to the representations appearing in Definition \ref{URPN}, let us consider the unitary representations $C_{V_\ell}$ and $C_{\hat V_\ell}$ of $H_\ell$ on
$L^2(V_\ell)$ and $L^2(\hat V_\ell)$, respectively, and the unitary representation $C_{Q_\ell}$ of~$P_{\ell-1}$ on~$L^2 (Q_\ell )$ defined by
\begin{align*}
C_{V_\ell}(h_\ell)f(v_\ell)&:=|h_\ell|_{V_\ell}^{-1/2}\,f({\bf C}_{h_\ell^{-1}}(v_\ell)),\\
C_{\hat V_\ell}(h_\ell)f(\xi_\ell)&:=|h_\ell|_{V_\ell}^{1/2}\,f({h_\ell^{-1}}^\flat\xi_\ell),\\
C_{Q_\ell}(p_{\ell-1})f(q_\ell)&:=|p_{\ell-1}|_{Q_\ell}^{-1/2}\,f({\bf C}_{p_{\ell-1}^{-1}}(q_\ell)).
\end{align*}
Lastly, we let $U_{\tilde\alpha,N_{\ell-1}}$ and $U_{\tilde\beta,Q_\ell}$ be the unitary representations of $Q_\ell$  on $L^2(N_{\ell-1})$  and  of $N_{\ell-1}$  on $L^2(Q_\ell )$ defined by
\begin{align*}
U_{\tilde\alpha,N_{\ell-1}}( q_\ell)f(n_{\ell-1})&:=|q_\ell|_{V_\ell}^{1/2}\,J_{\tilde\alpha_\ell}(q_\ell^{-1},n_{\ell-1})^{1/2}\, f\big(\tilde\alpha_{q_\ell^{-1}}(n_{\ell-1})\big),\\
U_{\tilde\beta,Q_\ell}(n_{\ell-1} )f(q_\ell)&:= f\big(\tilde\beta_{n_{\ell-1}^{-1}}(q_\ell)\big).
\end{align*}

Since the map $\phi_\ell:Q_\ell\to\hat V_\ell$ intertwines the left action of $Q_\ell$ on itself with the dual action on~$\hat V_\ell$,
 we deduce
 \begin{align}
 \label{CR1}
  \lambda_{q_\ell}\, U_{\phi_\ell}\,\F_{V_\ell}= U_{\phi_\ell}\,C_{\hat V_\ell}(q_\ell)\F_{V_\ell}= U_{\phi_\ell}\,\F_{V_\ell}\, C_{V_\ell}(q_\ell).
  \end{align}
This already proves the first relation for   $\ell=1$,  since in this case we have $U_{\tilde \alpha_1}=C_{V_1}$ and  $\CF_1=V_{\phi_1}\,\F_{V_1}$.

To prove the first commutation relation for all $\ell$, we proceed by induction and first consider the case where $p_\ell$ belongs to $Q_\ell$. From
the identities $ \lambda_{  q_\ell}V_{1,\ell}^*=V_{1,\ell}^*( \lambda_{q_\ell}\otimes 1)$, $( \lambda_{q_\ell}\otimes 1)  V_{\tilde\alpha_\ell} = V_{\tilde\alpha_\ell} (
  \lambda_{q_\ell}\otimes U_{\tilde\alpha,N_{\ell-1}}(q_\ell))$ and \eqref{CR1}, we get
 $$
 \lambda_{  q_\ell}\,\F_\ell=V_{1,\ell}^*\,(1\otimes \F_{\ell-1})\, V_{\tilde\alpha_\ell} \,
 (  U_{\phi_\ell}\,\F_{V_\ell}\otimes 1)(C_{V_\ell}(q_\ell)\otimes U_{\tilde\alpha,N_{\ell-1}}(q_\ell))\,V_{2,\ell},
$$
 and this is what we need,  because we have
$ \tilde\alpha_{ p_\ell^{-1}}(n_\ell)
={\bf C}_{ q_\ell^{-1}}(v_\ell)\tilde\alpha_{q_\ell^{-1}}(n_{\ell-1})$ by \eqref{TRL}.

Next, consider the case where  $p_\ell$ belongs to $P_{\ell-1}$.  We have
$\lambda_{p_{\ell-1}}V_{1,\ell}^*=V_{1,\ell}^*(C_{Q_\ell}(p_{\ell-1})\otimes \lambda_{p_{\ell-1}})$,
and thus we get by the induction hypothesis:
$$
\lambda_{p_{\ell-1}}\,\F_\ell=V_{1,\ell}^*\,(1\otimes \F_{\ell-1})\, (C_{Q_\ell}(p_{\ell-1})\otimes U_{\tilde\alpha_{\ell-1}}(p_{\ell-1}))\,
 V_{\tilde\alpha_\ell} \,( U_{\phi_\ell}\,\F_{V_\ell}\otimes 1)\,V_{2,\ell}.
$$
It is easy to see that  $ C_{Q_\ell} (p_{\ell-1})\otimes U_{\tilde\alpha_{\ell-1}}(p_{\ell-1})$ commutes with
$ V_{\tilde\alpha_\ell} $, and since ${p_{\ell-1}^{-1}}^\flat\xi_{0,\ell}=\xi_{0,\ell}$, we have
$ C_{Q_\ell} (p_{\ell-1})\, U_{\phi_\ell}\,\F_{V_\ell}= U_{\phi_\ell}\, C_{\hat V_\ell} (p_{\ell-1})\,\F_{V_\ell}= U_{\phi_\ell}\,\F_{V_\ell}\, C_{V_\ell}(p_{\ell-1})$. Hence we obtain:
 $$
\lambda_{p_{\ell-1}}\,\F_\ell=V_{1,\ell}^*\,(1\otimes \F_{\ell-1})\,
 V_{\tilde\alpha_\ell} \,( U_{\phi_\ell}\,\F_{V_\ell}\otimes 1)( C_{Q_\ell} (p_{\ell-1})\otimes U_{\tilde\alpha_{\ell-1}}(p_{\ell-1}))\,V_{2,\ell},
$$
and we conclude again by \eqref{TRL}.

\smallskip

 Let us now prove the second commutation relation. When $\ell=1$, the result follows  from
$U_{\tilde \beta_1}(v_1)={\rm Id}$ and $\E_1(q_1,v_1)=e^{i\langle\phi_1(q_1),v_1\rangle}$. We then proceed by induction on $\ell$ and first consider the case where
 $n_\ell$ belongs to $V_\ell$. From the identities $V_{2,\ell}\,\lambda_{v_\ell}=(\lambda_{v_\ell}\otimes 1)\,V_{2,\ell}$,
$ U_{\phi_\ell}\,\F_{V_\ell}\,\lambda_{v_\ell}= e^{-i\langle\phi_\ell(\cdot),v_\ell\rangle}\, U_{\phi_\ell}\,\F_{V_\ell}$
and $ V_{\tilde\alpha_\ell} \,(e^{-i\langle\phi_\ell(\cdot),v_\ell\rangle}\otimes 1)=(e^{-i\langle\phi_\ell(\cdot),v_\ell\rangle}\otimes 1)\,  V_{\tilde\alpha_\ell} $,
we deduce that
$$
\F_\ell\,\lambda_{v_\ell}=V_{1,\ell}^*\,(e^{i\langle\phi_\ell(\cdot),v_\ell\rangle}\otimes 1)\,(1\otimes \F_{\ell-1})\, V_{\tilde\alpha_\ell} \,( U_{\phi_\ell}\,\F_{V_\ell}\otimes 1)\,V_{2,\ell},
$$
and the result follows, because  $\E_\ell(p_\ell, v_\ell)=e^{i\langle\phi_\ell(q_\ell),v_\ell\rangle}$.

Next, we consider the case where  $n_\ell$ belongs to $N_{\ell-1}$.
The relations $V_{2,\ell}\,\lambda_{n_{\ell-1}}=( C_{V_\ell} (n_{\ell-1})\otimes \lambda_{n_{\ell-1}})\,V_{2,\ell}$ and $ U_{\phi_\ell}\,\F_{V_\ell}\, C_{V_\ell} (n_{\ell-1})=
 U_{\phi_\ell}\, C_{\hat V_\ell} (n_{\ell-1})\,\F_{V_\ell}=U_{\tilde\beta,Q_\ell}(n_{\ell-1})\, U_{\phi_\ell}\,\F_{V_\ell}$ give
$$
\F_\ell\,\lambda_{n_{\ell-1}}=V_{1,\ell}^*\,(1\otimes \F_{\ell-1})\, V_{\tilde\alpha_\ell} \,(U_{\tilde\beta,Q_\ell}(n_{\ell-1})
\otimes \lambda_{n_{\ell-1}})\,( U_{\phi_\ell}\,\F_{V_\ell}\otimes 1)\,V_{2,\ell}.
$$
Observe now that
$$
 V_{\tilde\alpha_\ell} \,(U_{\tilde\beta,Q_\ell}( n_{\ell-1})
\otimes \lambda_{ n_{\ell-1}})f(\tilde q_\ell,\tilde n_{\ell-1})= |q_\ell|_{V_\ell}^{-1/2}\,
J_{\tilde\gamma_\ell}(\tilde q_\ell,n_{\ell-1})^{1/2}\,f\big(\tilde\beta_{n_{\ell-1}^{-1}}(\tilde q_\ell),n_{\ell-1}^{-1}\tilde\alpha_{\tilde q_\ell}(\tilde n_{\ell-1}) \big).
$$
Since
$$
n_{\ell-1}^{-1}\tilde\alpha_{\tilde q_\ell}(\tilde n_{\ell-1})=\tilde\alpha_{\tilde\beta_{n_{\ell-1}^{-1}}(\tilde q_\ell)}\big(\tilde\alpha_{\tilde q_\ell^{-1}}(n_{\ell-1})^{-1}\tilde n_{\ell-1}\big),
$$
we deduce that
$$
 V_{\tilde\alpha_\ell} \,(U_{\tilde\beta,Q_\ell}( n_{\ell-1})
\otimes \lambda_{ n_{\ell-1}})f(\tilde q_\ell,\tilde n_{\ell-1})= V_{\tilde\alpha_\ell} f\big(\tilde\beta_{n_{\ell-1}^{-1}}(\tilde q_\ell),\tilde\alpha_{\tilde q_\ell^{-1}}(n_{\ell-1})^{-1}\tilde n_{\ell-1} \big).
$$
Therefore we get by the induction hypothesis that
\begin{multline*}
(1\otimes \F_{\ell-1})\, V_{\tilde\alpha_\ell} \,(U_{\tilde\beta,Q_\ell}( n_{\ell-1})\otimes \lambda_{ n_{\ell-1}})f(q_\ell,p_{\ell-1})=\\
\overline\E_{\ell-1}\big(p_{\ell-1}, \tilde\alpha_{q_\ell^{-1}}(n_{\ell-1})\big)\,
(1\otimes \F_{\ell-1})\, V_{\tilde\alpha_\ell} f\big(\tilde\beta_{n_{\ell-1}^{-1}}(q_\ell),\tilde\beta_{\tilde\alpha_{ q_\ell^{-1}}(n_{\ell-1})^{-1} }(p_{\ell-1})\big).
\end{multline*}
A simple computation shows that
$$
U_{\tilde \beta_\ell}(n_{\ell-1})\,V_{1,\ell}^*f(p_\ell q_{\ell-1})=f\big(\tilde\beta_{n_{\ell-1}^{-1}}(q_\ell),\tilde\beta_{\tilde\alpha_{ q_\ell^{-1}}(n_{\ell-1})^{-1} }(p_{\ell-1})\big),
$$
and the conclusion follows from the equality $\E_{\ell-1}\big(p_{\ell-1}, \tilde\alpha_{q_\ell^{-1}}(n_{\ell-1})\big)=\E_\ell(p_\ell,n_{\ell-1})$ proven in Lemma \ref{RecE}.
\end{proof}

\begin{remark}
From Proposition \ref{CRF} we can deduce that the map $n_\ell\mapsto \overline\E_\ell(\cdot,n_\ell)\,U_{\tilde \beta}(n_\ell)$ also defines a representation of $N_\ell$ on $L^2(P_\ell)$.
This fact is equivalent to the second identity in Lemma \ref{SE},
and this implies that the Fourier kernel $\E_\ell$ is  a $1$-cocycle, that is, as an element of $Z^1(N_\ell;\mathcal U(L^\infty(P_\ell)))$, where the action of $N_\ell$  on $L^\infty(P_\ell)$
is given by $\Ad\,U_{\tilde\beta}$.
\end{remark}

\subsection{The representation}

By Lemma~\ref{lem:typeI} we already   know that $G_\ell$ possesses a single  class of square-integrable
irreducible unitary representations. By the discussion at the beginning of Section~\ref{ssec:class}, a representative of this class is given by the Mackey representation
$$
\tilde\pi_\ell:={\Ind}_{G_{\ell-1}\ltimes V_\ell}^{G_\ell}(\tilde\pi_{\ell-1}\otimes \xi_{0,\ell}).
$$
However, this representative is not suitable for us and instead we consider another induced representation
\begin{align}
\label{TheRep}
\pi_\ell:={\Ind}_{N_{\ell}}^{G_\ell}(\chi_\ell),
\end{align}
where $\chi_\ell:N_\ell\to\T$ is the unitary character given in Definition \ref{Dchi}.

Since $(P_\ell,N_\ell)$ is a matched pair for $G_\ell$, it is natural to realize the representation $\pi_\ell$ on $L^2(P_\ell)$.
\begin{lemma}
For $\vf\in L^2(P_\ell)$ and $(p_\ell,n_\ell)\in P_\ell\times N_\ell$, we have
$$
\pi_\ell(p_\ell n_\ell)\vf(\tilde p_\ell)=
\E_\ell(p_\ell^{-1}\tilde p_\ell ,n_\ell)\,
\vf\big( \tilde\beta_{n_{\ell}^{-1}}(p_\ell^{-1}\tilde p_\ell)\big).
$$
Equivalently, in terms of the representation $U_{\tilde\beta}$ of $N_\ell$ given  in Definition \ref{URPN},  we have:
\begin{equation}
\label{GPI}
\pi_\ell(p_\ell n_\ell)=\lambda_{p_\ell}\,\E_\ell(\cdot ,n_\ell)\,U_{\tilde\beta}(n_\ell).
\end{equation}
\end{lemma}
\begin{proof}
This is an immediate consequence of Corollary \ref{C1P1} and of the formula
$$
(p_\ell n_\ell)^{-1}\tilde p_\ell=\tilde\beta_{n_{\ell}^{-1}}(p_\ell^{-1}\tilde p_\ell)\,\tilde \alpha_{\tilde p_\ell^{-1}p_\ell}(n_\ell)^{-1},
$$
where $p_\ell,\tilde p_\ell\in P_\ell$ and $n_\ell\in N_\ell$.
\end{proof}

\begin{proposition}
The representations $\pi_\ell$ and  $\tilde\pi_\ell$ are unitarily equivalent.
\end{proposition}
\begin{proof}
We realize $\pi_\ell$ on $\CH_\ell:=L^2(P_\ell)$ as in the previous lemma, and we realize $\tilde \pi_\ell$ on the Hilbert space $\tilde\CH_\ell$ inductively defined by
$\tilde\CH_1:=L^2(Q_1)$ and $\tilde\CH_\ell:=L^2(Q_\ell, \tilde\CH_{\ell-1})$. Starting from  the identities
$$
p_\ell^{-1}\,\tilde q_\ell={\bf C}_{p_{\ell-1}^{-1}}(q_\ell^{-1}\tilde q_\ell )\,p_{\ell-1}^{-1},
$$
and
$$
n_\ell^{-1}\,\tilde q_\ell=n_{\ell-1}^{-1}\,\tilde q_\ell\,{\bf C}_{\tilde q_\ell^{-1}}(v_\ell)^{-1}=\tilde\beta_{n_{\ell-1}^{-1}}(\tilde q_\ell )\,\tilde\alpha_{\tilde q_\ell^{-1}}(n_{\ell-1})^{-1}
\,{\bf C}_{\tilde q_\ell^{-1}}(v_\ell)^{-1},
$$
we get for $\vf\in \tilde\CH_\ell$:
\begin{align}
\label{tildepiP}
\big(\tilde\pi_\ell(p_\ell)\vf\big)(\tilde q_\ell)= |p_{\ell-1}|_{Q_\ell}^{-1/2}\,\tilde\pi_{\ell-1}(p_{\ell-1})\big(\vf\big({\bf C}_{p_{\ell-1}^{-1}}(q_\ell^{-1}\tilde q_\ell )\big)\big),
\end{align}
and, using Corollary \ref{C2L1}, we also get
\begin{align}
\label{tildepiN}
\big(\tilde\pi_\ell(n_\ell)\vf\big)(\tilde q_\ell)=  e^{i \langle\phi_{0,\ell}(\tilde q_\ell),v_\ell\rangle} \,
\,\tilde\pi_{\ell-1}\big(\tilde\alpha_{\tilde q_\ell ^{-1}}(n_{\ell-1})\big)\big(\vf\big(\tilde\beta_{n_{\ell-1}^{-1}}(\tilde q_\ell )\big)\big).
\end{align}

We will now show that $\pi_\ell$ and $\tilde\pi_\ell$ are unitarily equivalent by induction on $\ell$. For $\ell=1$ we clearly have $\tilde\CH_1=\CH_1$ and  $\pi_1=\tilde\pi_1$. So, assume that $\pi_{\ell-1}$ and $\tilde\pi_{\ell-1}$ are unitarily equivalent. In order to simplify the notation we then identify $\tilde\CH_{\ell-1}$ with $L^2(P_{\ell-1})$ in such a way that $\tilde\pi_{\ell-1}=\pi_{\ell-1}$. Let $U\colon L^2(P_\ell)\to L^2(Q_\ell)\otimes L^2(P_{\ell-1})$ be the unitary operator given by
$$
(Uf)(q_\ell,p_{\ell-1}):=|p_{\ell-1}|_{Q_\ell}^{-1/2}f(q_\ell,p_{\ell-1}).
$$
Take  $\vf\in L^2(Q_\ell)$ and   $\vf'\in L^2(P_{\ell-1})$. Then we have for $p_\ell\in P_\ell$ that
$$
\big(U\pi_\ell(p_\ell)U^*(\vf\otimes \vf')\big)(\tilde q_\ell,\tilde p_{\ell-1})=|p_{\ell-1}|_{Q_\ell}^{-1/2}\,\vf\big({\bf C}_{p_{\ell-1}^{-1}}(q_\ell^{-1}\tilde q_\ell )\big)\,\vf'(p_{\ell-1}^{-1}\tilde p_{\ell-1}),
$$
which under the identifications $ L^2(Q_\ell)\otimes L^2(P_{\ell-1})= L^2\big(Q_\ell, L^2(P_{\ell-1})\big)$ and $\tilde\pi_{\ell-1}=\pi_{\ell-1}$ is just
the expression in \eqref{tildepiP}.
Next, we have for $n_\ell\in N_\ell$:
$$
\big(\pi_\ell(n_\ell)U^*(\vf\otimes \vf')\big)(\tilde p_\ell)
=
\E_\ell(\tilde p_\ell ,n_\ell)
\big(U^*(\vf\otimes \vf')\big)\big( \tilde\beta_{n_{\ell}^{-1}}(\tilde p_\ell)\big).
$$
Noting that
\begin{multline*}
 \tilde\beta_{n_{\ell}^{-1}}(\tilde p_\ell)=\tilde\beta_{n_{\ell-1}^{-1}}(\tilde p_\ell)=
  \beta_{n_{\ell-1}^{-1}}(\tilde p_{\ell-1}^{-1}\tilde q_\ell^{-1})^{-1}\\=
  \beta_{n_{\ell-1}^{-1}}(\tilde q_\ell^{-1})^{-1} \,\beta_{\tilde\alpha_{\tilde q_\ell^{-1}}(n_{\ell-1})^{-1}}(\tilde p_{\ell-1}^{-1})^{-1}
  =  \tilde\beta_{n_{\ell-1}^{-1}}(\tilde q_\ell)\,\tilde \beta_{\tilde\alpha_{\tilde q_\ell^{-1}}(n_{\ell-1})^{-1}}(\tilde p_{\ell-1}),
 \end{multline*}
and using  Lemma \ref{RecE}, we get:
\begin{multline*}
\big(U\pi_\ell(n_\ell)U^*(\vf\otimes \vf')\big)(\tilde p_\ell,\tilde q_{\ell-1})\\
=
e^{i\langle \phi_\ell(\tilde q_\ell),v_\ell\rangle}\,\E_{\ell-1}\big(\tilde p_{\ell-1} ,\tilde\alpha_{\tilde q_\ell^{-1}}(n_{\ell-1})\big)
\vf\big( \tilde\beta_{n_{\ell-1}^{-1}}(\tilde q_\ell)\big)\,
\vf'\big(\tilde \beta_{\tilde\alpha_{\tilde q_\ell^{-1}}(n_{\ell-1})^{-1}}(\tilde p_{\ell-1})\big)\\
=
e^{i\langle \phi_\ell(\tilde q_\ell),v_\ell\rangle}\,\vf\big( \tilde\beta_{n_{\ell-1}^{-1}}(\tilde q_\ell)\big)\,
\pi_{\ell-1}\big(\tilde\alpha_{\tilde q_\ell^{-1}}(n_{\ell-1})\big)\vf'(\tilde p_{n-1}),
\end{multline*}
where we used the invariance property $|\beta_{n_{\ell-1}}(p_{\ell-1})|_{Q_\ell}=|p_{\ell-1}|_{Q_\ell}$, which follows from Lemma \ref{IDMF} and Remark~\ref{InvM}.
Under the same identifications as before, this is just the expression in~\eqref{tildepiN}. Therefore $U\pi_\ell(\cdot)U=\tilde\pi_\ell$.
\end{proof}

\begin{corollary}
\label{DMC}
The unitary representation $\pi_\ell$ is irreducible and square-integrable. Moreover, the Duflo--Moore operator $D_\ell$ is the densely defined operator on $L^2(P_\ell)$ given by multiplication
by the function $\Delta_{G_\ell}\big|_{P_\ell}^{-1}$.
\end{corollary}
\begin{proof}
Irreducibility and square integrability of $\pi_\ell$ follow from the unitary equivalence with the  Mackey representation $\tilde\pi_\ell$.
The Duflo--Moore operator $D_\ell$ of the square-integrable representation~$\pi_\ell$
is characterized as  the unique semi-invariant operator of weight $\Delta_{G_\ell}$, see \cite{DM}*{Theorem 3}, but note that we use the opposite conventions:
$$
\pi_\ell(g_\ell)\,D_\ell\,\pi_\ell(g_\ell)^*=\Delta_{G_\ell}(g_\ell)\,D_\ell.
$$
Thus it is enough to prove this identity for the operator given by multiplication by the function $\Delta_G\big|_{P_\ell}^{-1}$. For $\vf\in C_c(P_\ell)$, we get from
the expression \eqref{GPI} that
$$
\pi_\ell(g_\ell)\big(\Delta_{G_\ell}\big|_{P_\ell}^{-1}\vf\big)(\tilde p_\ell)=\Delta_{G_\ell}\big(\tilde\beta_{n_{\ell}^{-1}}(p_\ell^{-1}\tilde p_\ell)\big)^{-1}\pi_\ell(g_\ell)\vf(\tilde p_\ell),
$$
which concludes the proof, since by Lemma \ref{C1L1} the function $\Delta_{G_\ell}\big|_{P_\ell}$ is $\tilde\beta$-invariant.
\end{proof}

\subsection{The quantization map}
We are  ready to introduce the Kohn--Nirenberg type quantization $\Op_\ell$ of $G_\ell$.
To motivate the construction, we start with formal considerations.
Consider  the    Radon measures on $P_\ell$ defined for $\vf\in C_c(P_\ell)$ by
$$
 T_1(\vf):=\vf(e)
 \quad\mbox{and}\quad T_2(\vf):=\int_{P_\ell} \vf(p_\ell)\,
 J_{\tilde\alpha_\ell}(p_{\ell}^{-1},e)^{1/2}\,dp_\ell.
$$
For $f\in C_c(G_\ell)$,
consider the (formal) sesquilinear form on $C_c(P_\ell)$ defined by
\begin{align}
\label{KNMQ}
\widetilde\Op_\ell(f)[\vf_1,\vf_2]:=\int_{G_\ell} f(g_\ell)\,\overline{T_1( \pi_\ell(g_\ell)^*\vf_1)}  \,T_2(\pi_\ell(g_\ell)^*\vf_2)
\, dg_\ell.
\end{align}
Explicitly, note first that by  \eqref{GPI} we have:
$$
T_1( \pi_\ell(p_\ell n_\ell)^*\vf_1)=U_{\tilde\beta}(n_\ell^{-1})\,\overline{\E_\ell}(\cdot ,n_\ell)\,\lambda_{p_\ell^{-1}}\vf_1(e)
=\overline{\E_\ell}(e,n_\ell)\,\vf_1(p_\ell)=\overline{\chi_\ell}(n_\ell)\,\vf_1(p_\ell).
$$
Then, using Corollary \ref{C2L1} and the definition of $J_{\tilde\alpha_\ell}$ given in Remark \ref{IF1}, we get
\begin{multline*}
T_2(\pi_\ell(g_\ell)^*\vf_2)=\int_{P_\ell} \pi_\ell(p_\ell n_\ell)^*\vf(\tilde p_\ell)\,
 J_{\tilde\alpha_\ell}(\tilde p_{\ell}^{-1},e)^{1/2}\,dp_\ell\\
 =\int_{P_\ell} \vf_2(\tilde p_\ell) \,\overline{\E_\ell}(p_\ell^{-1}\tilde p_\ell,n_\ell)\,
 J_{\tilde\alpha_\ell}(\tilde p_{\ell}^{-1} p_\ell,n_\ell)^{1/2}\,d\tilde p_\ell.
\end{multline*}
Hence by Corollary \ref{C2P1}, we get
$$
\widetilde\Op_\ell(f)[\vf_1,\vf_2]=\int_{P_\ell\times P_\ell}\overline{\vf_1}( p_\ell) \, K_\ell(f)( p_\ell,\tilde p_\ell)\,\vf_2(\tilde p_\ell)
\,dp_\ell d\tilde p_\ell,
$$
where the operator  kernel is given by
$$
K_\ell(f)( p_\ell,\tilde p_\ell)=\int_{N_\ell} \,\chi_\ell(n_\ell)\,f(p_\ell n_\ell)\,\overline{\E_\ell}(p_\ell^{-1}\tilde p_\ell,n_\ell)\,
J_{\tilde\alpha_\ell}(\tilde p_{\ell}^{-1} p_\ell,n_\ell)^{1/2} \, dn_\ell.
$$
This formal expression and Remark \ref{FFD} justify the following definition, which should be viewed as a central result of this paper:
\begin{definition}
\label{UQM}
Consider the unitary operator
$$
K_\ell:=W_{P_\ell}\,(1\otimes \F_\ell\,\chi_\ell)\,T_\ell :L^2(G_\ell)\to L^2(P_\ell \times P_\ell),
$$
where $T_\ell\colon L^2(G_\ell)\to L^2(P_\ell\times N_\ell)$  is  given by $(T_\ell f)(p_\ell,n_\ell):=f(p_\ell n_\ell)$ and
 $W_{P_\ell}$ is the multiplicative unitary of $P_\ell$.
We define the quantization map as the unitary operator
$$
\Op_\ell\colon L^2(G_\ell)\to\HS(L^2(P_\ell))
$$
which maps a function $f\in L^2(G_\ell)$ to
the Hilbert--Schmidt operator on $L^2(P_\ell)$ with the  operator kernel $K_\ell(f)\in L^2(P_\ell\times P_\ell)$.
\end{definition}
Of course, the formal definition \eqref{KNMQ} implies that $\Op_\ell$  intertwines $\lambda$ with ${\rm Ad}\,\pi_\ell$. The following theorem gives a rigorous proof of this property.

\begin{theorem}
The Kohn--Nirenberg quantization map $\Op_\ell$ intertwines the regular representation $\lambda$ of $G_\ell$ with ${\rm Ad}\,\pi_\ell$. Equivalently, the unitary operator
$K_\ell\colon L^2(G_\ell)\to L^2(P_\ell)\otimes L^2(P_\ell)$
intertwines $\lambda$ with $\pi_\ell\otimes \pi_\ell^c$, where $(\pi_\ell^c(g)\varphi)(p_\ell):=\overline{(\pi_\ell(g)\overline\varphi)(p_\ell)}$.
\end{theorem}
\begin{proof}
Due to the identities $\pi_\ell(p_\ell)=\pi_\ell^c(p_\ell)=\lambda_{p_\ell}$, we have
\begin{multline*}
(\pi_\ell(p_\ell)\otimes \pi_\ell^c(p_\ell))\, K_\ell=(\lambda_{p_\ell}\otimes \lambda_{p_\ell})W_{P_\ell}\,(1\otimes \F_\ell\,\chi_\ell)\,T_\ell
=W_{P_\ell}\,(\lambda_{p_\ell}\otimes1)\,(1\otimes \F_\ell\,\chi_\ell)\,T_\ell\\
=W_{P_\ell}\,(1\otimes \F_\ell\,\chi_\ell)\,(\lambda_{p_\ell}\otimes1)\,T_\ell
=W_{P_\ell}\,(1\otimes \F_\ell\,\chi_\ell)\,T_\ell\,\lambda_{p_\ell}=K_\ell\,\lambda_{p_\ell}.
\end{multline*}

Next, since $\pi_\ell(n_\ell)=\E_\ell(\cdot ,n_\ell)\,U_{\tilde\beta_\ell}(n_\ell)$, we get $\pi_\ell^c(n_\ell)=\overline{\E}_\ell(\cdot ,n_\ell)\,U_{\tilde\beta_\ell}(n_\ell)$, so that  for $f\in L^2(G_\ell)$
we obtain
\begin{multline*}
\big((\pi_\ell(n_\ell)\otimes \pi_\ell^c(n_\ell))\, K_\ell (f)\big)(p_\ell,\tilde p_\ell)
= \E_\ell(p_\ell,n_\ell)\, \overline\E_\ell(\tilde p_\ell,n_\ell)\, K_\ell(f)\big(\tilde\beta_{n_\ell^{-1}}(p_\ell),\tilde\beta_{n_\ell^{-1}}(\tilde p_\ell)\big)\\
= \E_\ell(p_\ell,n_\ell)\, \overline\E_\ell(\tilde p_\ell,n_\ell)\,\big((1\otimes \F_\ell\,\chi_\ell)\,T_\ell
f\big)\big(\tilde\beta_{n_\ell^{-1}}(p_\ell),\tilde\beta_{n_\ell^{-1}}(p_\ell)^{-1}\tilde\beta_{n_\ell^{-1}}(\tilde p_\ell)\big).
\end{multline*}
Observe that
$$
\tilde\beta_{n_\ell^{-1}}(p_\ell)^{-1}\tilde\beta_{n_\ell^{-1}}(\tilde p_\ell)=\tilde\beta_{\tilde\alpha_{p_\ell^{-1}}(n_\ell)^{-1}}(p_\ell^{-1}\tilde p_\ell).
$$
It follows that
\begin{multline*}
\big((\pi_\ell(n_\ell)\otimes \pi_\ell^c(n_\ell))\, K_\ell (f)\big)(p_\ell,\tilde p_\ell)\\
= \E_\ell(p_\ell,n_\ell)\, \overline\E_\ell(\tilde p_\ell,n_\ell)\,\big((1\otimes \F_\ell\,\chi_\ell)\,T_\ell
f\big)\big(\tilde\beta_{n_\ell^{-1}}(p_\ell),\tilde\beta_{\tilde\alpha_{p_\ell^{-1}}(n_\ell)^{-1}}(p_\ell^{-1}\tilde p_\ell)\big).
\end{multline*}
Looking at the second leg in the last expression, we are led to consider, for $p_\ell\in P_\ell$ fixed, the following unitary operator:
\begin{equation}
\label{UOP}
\E_\ell(p_\ell,n_\ell)\, \overline\E_\ell(\cdot,n_\ell)\,\lambda_{p_\ell}\,U_{\tilde\beta}\big(\tilde\alpha_{p_\ell^{-1}}(n_\ell)\big)\,\F_\ell\,\chi_\ell.
\end{equation}
By Proposition \ref{CRF} this operator coincides with
\begin{multline*}
\E_\ell(p_\ell,n_\ell)\, \overline\E_\ell(\cdot,n_\ell)\,\lambda_{p_\ell}\,\E\big(\cdot,\tilde\alpha_{p_\ell^{-1}}(n_\ell)\big)\,\F_\ell\,\lambda_{\tilde\alpha_{p_\ell^{-1}}(n_\ell)}\,\chi_\ell\\
=\E_\ell(p_\ell,n_\ell)\, \overline\E_\ell(\cdot,n_\ell)\,\E\big(p_\ell^{-1}\cdot,\tilde\alpha_{p_\ell^{-1}}(n_\ell)\big)\,\lambda_{p_\ell}\,\F_\ell\,
\lambda_{\tilde\alpha_{p_\ell^{-1}}(n_\ell)}\,\chi_\ell\\
=\E_\ell(p_\ell,n_\ell)\,\lambda_{p_\ell}\,\F_\ell\,
\lambda_{\tilde\alpha_{p_\ell^{-1}}(n_\ell)}\,\chi_\ell,
\end{multline*}
where the last equality follows from the first relation in  Lemma \ref{SE}. Using that
$$
\lambda_{\tilde\alpha_{p_\ell^{-1}}(n_\ell)}\,\chi_\ell=\chi_\ell\big(\tilde\alpha_{p_\ell^{-1}}(n_\ell)^{-1}\big)\,\chi_\ell\,\lambda_{\tilde\alpha_{p_\ell^{-1}}(n_\ell)}
=\overline\E_\ell(p_\ell,n_\ell)\,\chi_\ell\,\lambda_{\tilde\alpha_{p_\ell^{-1}}(n_\ell)},
$$
we see that the unitary operator \eqref{UOP} is equal to $\lambda_{p_\ell}\,\F_\ell\,\chi_\ell\,\lambda_{\tilde\alpha_{p_\ell^{-1}}(n_\ell)}$.
Hence we get
$$
\big((\pi_\ell(n_\ell)\otimes \pi_\ell^c(n_\ell))\, K_\ell (f)\big)(p_\ell,\tilde p_\ell)
=\big((1\otimes \lambda_{p_\ell}\,\F_\ell\,\chi_\ell\,\lambda_{\tilde\alpha_{p_\ell^{-1}}(n_\ell)})\,T_\ell
f\big)\big(\tilde\beta_{n_\ell^{-1}}(p_\ell),\tilde p_\ell\big).
$$
It follows that
$$
\big(K_\ell ^*\,(\pi_\ell(n_\ell)\otimes \pi_\ell^c(n_\ell))\, K_\ell(f)\big)(p_\ell\tilde n_\ell)=f\big(\tilde\beta_{n_\ell^{-1}}(p_\ell)\tilde\alpha_{p_\ell^{-1}}(n_\ell)^{-1}\tilde n_\ell\big)
=f(n_\ell^{-1}p_\ell\tilde n_\ell),
$$
which concludes the proof.
\end{proof}

\subsection{The dual cocycle}
By \cite{BGNT3} we know now that we have a dual unitary $2$-cocycle $\Omega_\ell$ on $G_\ell$ given by the formula
$$
\Omega_\ell:= (\CJ\otimes\CJ)\,\G^*\,  (1\otimes\CJ)\, \hat W_{G_\ell},
$$
 where $(\CJ f)(g_\ell)=\Delta_{G_\ell}(g_\ell)^{-1/2}\,f(g^{-1})$ and
  $\G\colon L^2(G_\ell)\otimes L^2(G_\ell)\to L^2(G_\ell)\otimes L^2(G_\ell)$ is the unitary Galois map. With  the Duflo--Moore operator $D_\ell$ of the representation
  $\pi_\ell$, this Galois map is given  by
$$
\big(\G(f_1\otimes f_2)\big)(g_\ell,\tilde g_\ell)=\Delta_{G_\ell}(g_\ell)^{-1/2}\Op_\ell^*\big(\Op_\ell(\lambda_{g_\ell}f_1)D_\ell^{-1/2}\Op_\ell(f_2)\big)(\tilde g_\ell).
$$
Let us convince ourselves, at least formally, that this dual $2$-cocycle can be written as follows:
\begin{equation}
\label{FF}
\Omega_\ell=\int_{P_\ell\times N_\ell}
\overline{\chi_\ell}(n_\ell)\,\E_\ell(p_\ell,n_\ell)\, J_{\tilde\alpha_\ell}( p_\ell^{-1},n_\ell)^{1/2}\,J_{\tilde\alpha_\ell}( p_\ell,e)^{1/2}
\,\lambda_{n_\ell^{-1}}\otimes \lambda_{p_\ell^{-1}}\, dp_\ell dn_\ell.
\end{equation}
Observe that if we let $\mathbb F_\ell(p_\ell,n_\ell):=J_{\tilde\alpha_\ell}( p_\ell^{-1},n_\ell)^{1/2}\,\E_\ell(p_\ell,n_\ell)$ be the total Fourier kernel
(see Remark \ref{FFD}), then we get
$$
\Omega_\ell=\int_{P_\ell\times N_\ell} \frac{\mathbb F_\ell(p_\ell,n_\ell)}{\mathbb F_\ell(e,n_\ell)\mathbb F_\ell(p_\ell,e)}\,\lambda_{n_\ell^{-1}}\otimes \lambda_{p_\ell^{-1}}\, dp_\ell dn_\ell.
$$
Note also that if we let $\Phi_\ell$ be the phase of this kernel, that is,
$$
\Phi_\ell(p_\ell,n_\ell):=\frac{\mathbb E_\ell(p_\ell,n_\ell)}{\mathbb E_\ell(e,n_\ell)},
$$
then this function satisfies the bi-$1$-cocycle   relations
\begin{equation*}
\Phi_\ell(p_\ell\tilde p_\ell,n_\ell)=\Phi_\ell(p_\ell,n_\ell)\,\Phi_\ell(\tilde p_\ell,\tilde\alpha_{p_\ell^{-1}}(n_\ell))\quad\mbox{and}\quad
\Phi_\ell(p_\ell, n_\ell\tilde n_\ell)=\Phi_\ell(p_\ell,n_\ell)\,\Phi_\ell(\tilde\beta_{n_\ell^{-1}}(p_\ell),\tilde n_\ell).
\end{equation*}

To obtain formula \eqref{FF}, we let $F$  be the  pseudo-measure on $G_\ell\times G_\ell$ such that
$$
\Omega_\ell^*=\int_{G_\ell\times G_\ell} F(g_\ell,\tilde g_\ell)\,\lambda_{g_\ell}\otimes \lambda_{\tilde g_\ell}\, dg_\ell d\tilde g_\ell.
$$
With $\delta_e$ the Dirac mass at the neutral element, we have
$$
F=\Omega_\ell^*(\delta_e\otimes \delta_e).
$$
Since $\CJ\delta_e=\delta_e$, we need to consider
$$
\big(\G(\delta_e\otimes \delta_e)\big)(g_\ell,\tilde g_\ell)
=
\Delta_{G_\ell}(g_\ell)^{-1/2}\Op_\ell^*\big(\Op_\ell(\lambda_{g_\ell}\delta_e)D_\ell^{-1/2}\Op_\ell(\delta_e)\big)(\tilde g_\ell).
$$
Observe that $\lambda_{g_\ell}(\delta_e)=\delta_{g_\ell}$ and, by Corollary \ref{C2P1}, we have $T_\ell (\delta_{p_\ell n_\ell})=\delta_{p_\ell}\otimes \delta_{n_\ell}$. Also, we have
$\chi_\ell(\delta_{n_\ell})=\chi_\ell(n_\ell)\,\delta_{n_\ell}$, and
$\CF_\ell(\delta_{n_\ell})=\overline{\E_\ell}(\cdot ,n_\ell)\, J_{\tilde\alpha_\ell}(\cdot^{-1},n_\ell)^{1/2}$ by Remark \ref{FFD}.
Noticing lastly that $W_{P_\ell}(\delta_{p_\ell}\otimes\vf)=\delta_{p_\ell}\otimes\lambda_{p_\ell}\vf$, we get
$$
K_\ell(\delta_{g_\ell})=\chi_\ell(n_\ell)\big(\delta_{p_\ell}\otimes \overline{\E_\ell}(p_\ell^{-1}\cdot ,n_\ell)\, J_{\tilde\alpha_\ell}((p_\ell^{-1}\cdot)^{-1},n_\ell)^{1/2}\big).
$$
In particular, we have
$$
K_\ell(\delta_{e})=\delta_{e}\otimes  J_{\tilde\alpha_\ell}(\cdot^{-1},e)^{1/2}.
$$
From this and Corollary \ref{DMC} we easily conclude that the kernel  of $\Op_\ell(\lambda_{g_\ell}\delta_e)D_\ell^{-1/2}\Op_\ell(\delta_e)$ is given by
$$
\chi_\ell(n_\ell)\,\overline{\E_\ell}(p_\ell^{-1},n_\ell)\, J_{\tilde\alpha_\ell}(p_\ell,n_\ell)^{1/2}\big(\delta_{p_\ell}\otimes
J_{\tilde\alpha_\ell}(\cdot^{-1},e)^{1/2}\big).
$$
Therefore we get
\begin{multline*}
\big(\G(\delta_e\otimes \delta_e)\big)(g_\ell,\tilde g_\ell)\\
=\chi_\ell(n_\ell)\,\overline{\E_\ell}(p_\ell^{-1},n_\ell)\, J_{\tilde\alpha_\ell}(p_\ell,n_\ell)^{1/2}\,
\Delta_{G_\ell}(p_\ell)^{-1/2}\, \big(K^*(\delta_{p_\ell}\otimes J_{\tilde\alpha_\ell}(\cdot^{-1},e)^{1/2})\big)(\tilde g_\ell).
\end{multline*}
Now, because $J_{\tilde\alpha}(\cdot,e)$ is a quasi-character on $P_\ell$, we have
$$
W^*_{P_\ell}(\delta_{p_\ell}\otimes J_{\tilde\alpha_\ell}(\cdot^{-1},e)^{1/2})=J_{\tilde\alpha_\ell}(p_\ell^{-1},e)^{1/2}(\delta_{p_\ell}\otimes J_{\tilde\alpha_\ell}(\cdot^{-1},e)^{1/2}).
$$
Since, moreover, $\CF_\ell^*(  J_{\tilde\alpha_\ell}(\cdot^{-1},e)^{1/2})=\delta_e$, we deduce
$$
\big(\G(\delta_e\otimes \delta_e)\big)(g_\ell,\tilde g_\ell)=\chi_\ell(n_\ell)\,\overline{\E_\ell}(p_\ell^{-1},n_\ell)\, J_{\tilde\alpha_\ell}(p_\ell,n_\ell)^{1/2}\,
\Delta_{G_\ell}(p_\ell)^{-1/2}\, J_{\tilde\alpha_\ell}(p_\ell^{-1},e)^{1/2}\,\delta_{p_\ell}(\tilde g_\ell).
$$
Finally, since $\delta_{p_\ell}(\tilde g_\ell^{-1})=\delta_{p_\ell}(\tilde p_\ell^{-1})\,\delta_e(\tilde n_\ell)$, we have
\begin{multline*}
\big(\G(\delta_e\otimes \delta_e)\big)(g_\ell,\tilde g_\ell)=\Delta_{G_\ell}(\tilde g_\ell)^{-1/2}\,\big(\G(\delta_e\otimes \delta_e)\big)(\tilde g_\ell^{-1}g_\ell,\tilde g_\ell^{-1})\\
=\chi_\ell(n_\ell)\,\overline{\E_\ell}(\tilde p_\ell,n_\ell)\, J_{\tilde\alpha_\ell}(\tilde p_\ell^{-1},n_\ell)^{1/2}\,J_{\tilde\alpha_\ell}(\tilde p_\ell,e)^{1/2}\,
\delta_e(p_\ell)\,\delta_e(\tilde n_\ell),
\end{multline*}
which yields formula \eqref{FF}. Whether one can make a rigorous sense of this formula and justify the above computations in concrete examples, depends on regularity properties of the function~$J_{\tilde\alpha_\ell}$.

\subsection{The semi-classical limit}

In this final section we discuss the semi-classical limit of a rescaled version $\Omega_\theta$, $\theta\in\R^*$, of the dual $2$-cocycle we have constructed,
in the simplest case when $G={\rm GL}_2(\R)\ltimes\R^2$. It is not difficult to see that in this case the modular function $\Delta_P$ is also $\beta$-invariant.
Therefore  the dual cocycle has an extremely simple form (cf.~\cite{BGNT3}):
$$
\Omega^*=\int_{P\times N}
\chi_\ell(n)\,\overline{\E}(p,n)
\,\lambda_{n}\otimes \lambda_{p}\, dp\, dn,
$$
where the subgroups $P$ and $N$ are given by
\begin{align*}
P:=\bigg\{\begin{pmatrix} \begin{pmatrix}*&*\\0&*\end{pmatrix},&\!\!\!\!\begin{pmatrix}0\\0\end{pmatrix}\end{pmatrix}\bigg\},
\quad
N:=\bigg\{\begin{pmatrix} \begin{pmatrix}1&0\\ *&1\end{pmatrix},&\!\!\!\!\begin{pmatrix}*\\ *\end{pmatrix}\end{pmatrix}\bigg\}.
\end{align*}
More explicitly, we have:
$$
\Omega^*=\int
e^{-i\langle q_2^\flat\xi_{0,2}-\xi_{0,2},v_2\rangle}\,
e^{-i\langle q_1^\flat\xi_{0,1},\tilde\alpha_{q_2^{-1}}(v_1)\rangle}\,e^{i\langle\xi_{0,1},v_1\rangle}\,
\lambda_{v_2v_1}\otimes\lambda_{q_2q_1}\frac{dq_2\, dq_1}{|q_1|_{Q_2}}dv_2\, dv_1,
$$
where  the
 subgroups $Q_2$, $Q_1$, $V_2$ and $V_1$, are given by
\begin{align*}
&Q_2:=\bigg\{\begin{pmatrix} \begin{pmatrix}*&*\\0&1\end{pmatrix},&\!\!\!\!\begin{pmatrix}0\\0\end{pmatrix}\end{pmatrix}\bigg\},
\quad Q_1:=\bigg\{\begin{pmatrix} \begin{pmatrix}1&0\\0&*\end{pmatrix},&\!\!\!\!\begin{pmatrix}0\\0\end{pmatrix}\end{pmatrix}\bigg\},
\\
&V_2:=\bigg\{\begin{pmatrix} \begin{pmatrix}1&0\\0&1\end{pmatrix},&\!\!\!\!\begin{pmatrix}*\\ *\end{pmatrix}\end{pmatrix}\bigg\},
\quad V_1:=\bigg\{\begin{pmatrix} \begin{pmatrix}1&0\\ *&1\end{pmatrix},&\!\!\!\!\begin{pmatrix}0\\0\end{pmatrix}\end{pmatrix}\bigg\}.
\end{align*}

Writing an element of $ G$ as
$g=\begin{pmatrix} \begin{pmatrix}a&b\\c&z\end{pmatrix},&\!\!\!\!\begin{pmatrix}x\\y\end{pmatrix}\end{pmatrix}$, the elements of $P=Q_1\ltimes Q_2$ and
$N=V_1\ltimes V_2$ are of the form
$$
p=\begin{pmatrix} \begin{pmatrix}a&b\\0&z\end{pmatrix},&\!\!\!\!\begin{pmatrix}0\\0\end{pmatrix}\end{pmatrix}\quad\mbox{and}\quad
n=\begin{pmatrix} \begin{pmatrix}1&0\\c&1\end{pmatrix},&\!\!\!\!\begin{pmatrix}x\\y\end{pmatrix}\end{pmatrix}.
$$
A simple calculation gives the following formulas for the dressing actions:
$$
\beta_n(p)=\begin{pmatrix} \begin{pmatrix}a-bc&b\\0&\frac{za}{a-bc}\end{pmatrix},&\!\!\!\!\begin{pmatrix}0\\0\end{pmatrix}\end{pmatrix}
\quad\mbox{and}\quad
\alpha_p(n)=\begin{pmatrix} \begin{pmatrix}1&0\\\frac{zc}{a-bc}&1\end{pmatrix},&\!\!\!\!\begin{pmatrix}(a-bc)x+by\\ \frac{za}{a-bc}y\end{pmatrix}\end{pmatrix}.
$$
Therefore we deduce:
\begin{multline*}
\Omega^*=\int
e^{-i((a^{-1}-1)x-a^{-1}by)}\,
e^{-i(z^{-1}\frac{ac}{1-bc}-c)}\cdot\\
\lambda_{\begin{pmatrix} \begin{pmatrix}1&0\\c&1\end{pmatrix},&\!\!\!\!\begin{pmatrix}x\\y\end{pmatrix}\end{pmatrix}}
\otimes\lambda_{\begin{pmatrix} \begin{pmatrix}a&bz\\0&z\end{pmatrix},&\!\!\!\!\begin{pmatrix}0\\0\end{pmatrix}\end{pmatrix}}
 \frac{da\,db}{a^2} \frac{dz\,dc}{z^2}dx\,dy,
\end{multline*}
an expression which after the change of variables $a\mapsto a^{-1}$ and $z\mapsto z^{-1}$ becomes
\begin{multline*}
\Omega^*=\int
e^{-i((a-1)x-aby)}\,
e^{-i(z\frac{a^{-1}c}{1-bc}-c)}\\
\lambda_{\begin{pmatrix} \begin{pmatrix}1&0\\c&1\end{pmatrix},&\!\!\!\!\begin{pmatrix}x\\y\end{pmatrix}\end{pmatrix}}
\otimes\lambda_{\begin{pmatrix} \begin{pmatrix}a^{-1}&bz^{-1}\\0&z^{-1}\end{pmatrix},&\!\!\!\!\begin{pmatrix}0\\0\end{pmatrix}\end{pmatrix}}
 da\,db\,dz\,dc\,dx\,dy.
\end{multline*}

Rescaling the elements $\xi_{0,2}\in \hat V_2$ and $\xi_{0,1}\in \hat V_1$ by  $\theta/2\pi$, $\theta\in\R^*$, we get a family of dual cocycles:
\begin{multline}
\Omega^*_\theta=\theta^{-3}\int
e^{-\frac{2i\pi}\theta((a-1)x-aby)}\,
e^{-\frac{2i\pi}\theta(z\frac{a^{-1}c}{1-bc}-c)}\\
\lambda_{\begin{pmatrix} \begin{pmatrix}1&0\\c&1\end{pmatrix},&\!\!\!\!\begin{pmatrix}x\\y\end{pmatrix}\end{pmatrix}}
\otimes\lambda_{\begin{pmatrix} \begin{pmatrix}a^{-1}&bz^{-1}\\0&z^{-1}\end{pmatrix},&\!\!\!\!\begin{pmatrix}0\\0\end{pmatrix}\end{pmatrix}}
 da\,db\,dz\,dc\,dx\,dy.
 \end{multline}
 By the change of variables $c\mapsto\theta c$, $x\mapsto\theta x$ and $y\mapsto\theta y$, this becomes
 \begin{multline*}
\Omega^*_\theta =\int
e^{-2i\pi((a-1)x-aby)}\,
e^{-2i\pi(z\frac{a^{-1}c}{1-\theta bc}-c)}\\
\lambda_{\begin{pmatrix} \begin{pmatrix}1&0\\\theta c&1\end{pmatrix},&\!\!\!\!\begin{pmatrix}\theta x\\\theta y\end{pmatrix}\end{pmatrix}}
\otimes\lambda_{\begin{pmatrix} \begin{pmatrix}a^{-1}&bz^{-1}\\0&z^{-1}\end{pmatrix},&\!\!\!\!\begin{pmatrix}0\\0\end{pmatrix}\end{pmatrix}}
 da\,db\,dz\,dc\,dx\,dy.
\end{multline*}
We proceed with a formal Taylor expansion of $\Omega^*_\theta$  of the first order  in a neighborhood of  $\theta=0$. For this, we need to consider the following vector fields:
\begin{align*}
&A:=\frac d{dt} \lambda_{\begin{pmatrix} \begin{pmatrix}e^t&0\\0&1\end{pmatrix},&\!\!\!\!\begin{pmatrix}0\\0\end{pmatrix}\end{pmatrix}}\Bigg|_{t=0},\;
B:=\frac d{dt} \lambda_{\begin{pmatrix} \begin{pmatrix}1&t\\0&1\end{pmatrix},&\!\!\!\!\begin{pmatrix}0\\0\end{pmatrix}\end{pmatrix}}\Bigg|_{t=0},\;
C:=\frac d{dt} \lambda_{\begin{pmatrix} \begin{pmatrix}1&0\\t&1\end{pmatrix},&\!\!\!\!\begin{pmatrix}0\\0\end{pmatrix}\end{pmatrix}}\Bigg|_{t=0},\\
&Z:=\frac d{dt} \lambda_{\begin{pmatrix} \begin{pmatrix}1&0\\0& e^t\end{pmatrix},&\!\!\!\!\begin{pmatrix}0\\0\end{pmatrix}\end{pmatrix}}\Bigg|_{t=0},\;
X:=\frac d{dt} \lambda_{\begin{pmatrix} \begin{pmatrix}1&0\\0&1\end{pmatrix},&\!\!\!\!\begin{pmatrix}t\\0\end{pmatrix}\end{pmatrix}}\Bigg|_{t=0},\;
Y:=\frac d{dt} \lambda_{\begin{pmatrix} \begin{pmatrix}1&0\\0&1\end{pmatrix},&\!\!\!\!\begin{pmatrix}0\\ t\end{pmatrix}\end{pmatrix}}\Bigg|_{t=0}.
\end{align*}
We then get
\begin{align*}
\Omega^*_\theta&=(1\otimes 1)\int \bigg(\int e^{-2i\pi((a-1)x-aby)}\, e^{-2i\pi(za^{-1}-1)c}
dc\,dx\,dy\bigg) da\,db\,dz\\
&+ \theta\, C\otimes\int \bigg(\int c\, e^{-2i\pi((a-1)x-aby)}\,
e^{-2i\pi(za^{-1}-1)c}dc\,dx\,dy\bigg)\lambda_{\begin{pmatrix} \begin{pmatrix}a^{-1}&bz^{-1}\\0&z^{-1}\end{pmatrix},&\!\!\!\!\begin{pmatrix}0\\0\end{pmatrix}\end{pmatrix}}
 da\,db\,dz\\
&+  \theta\, X\otimes\int\bigg(\int x\, e^{-2i\pi((a-1)x-aby)}\,
e^{-2i\pi(za^{-1}-1)c}dc\,dx\,dy\bigg)\lambda_{\begin{pmatrix} \begin{pmatrix}a^{-1}&bz^{-1}\\0&z^{-1}\end{pmatrix},&\!\!\!\!\begin{pmatrix}0\\0\end{pmatrix}\end{pmatrix}}
 da\,db\,dz\\
&+  \theta\, Y\otimes\int\bigg(\int y\, e^{-2i\pi((a-1)x-aby)}\,
e^{-2i\pi(za^{-1}-1)c}dc\,dx\,dy\bigg)\lambda_{\begin{pmatrix} \begin{pmatrix}a^{-1}&bz^{-1}\\0&z^{-1}\end{pmatrix},&\!\!\!\!\begin{pmatrix}0\\0\end{pmatrix}\end{pmatrix}}
 da\,db\,dz\\
&-2i\pi \theta\otimes\int \frac{zb}a \bigg(\int c^2
e^{-2i\pi((a-1)x-aby)}\,
e^{-2i\pi(za^{-1}-1)c}dc\,dx\,dy\bigg)
\lambda_{\begin{pmatrix} \begin{pmatrix}a^{-1}&bz^{-1}\\0&z^{-1}\end{pmatrix},&\!\!\!\!\begin{pmatrix}0\\0\end{pmatrix}\end{pmatrix}}
 da\,db\,dz\\
&+ O(\theta^2).
\end{align*}
Computing the Fourier transforms  (in the sense of tempered distributions), we get
\begin{align*}
\Omega^*_\theta&=(1\otimes 1)\int\delta_1(a)\,\delta_0(b)\,\delta_1(z)\, da\,db\,dz\\
&+ \frac{i\theta}{2\pi}\, C\otimes\int \delta_1(a)\,\delta_0(b)\,\delta'_1(z)\,\lambda_{\begin{pmatrix} \begin{pmatrix}a^{-1}&bz^{-1}\\0&z^{-1}\end{pmatrix},
&\!\!\!\!\begin{pmatrix}0\\0\end{pmatrix}\end{pmatrix}}
 da\,db\,dz\\
&+  \frac{i\theta}{2\pi}\, X\otimes\int\delta'_1(a)\,\delta_0(b)\,\delta_1(z)\,\lambda_{\begin{pmatrix} \begin{pmatrix}a^{-1}&bz^{-1}\\0&z^{-1}\end{pmatrix},
&\!\!\!\!\begin{pmatrix}0\\0\end{pmatrix}\end{pmatrix}}
 da\,db\,dz\\
&+  \frac{i\theta}{2\pi}\, Y\otimes\int\delta_1(a)\,\delta'_0(b)\,\delta_1(z)\,\lambda_{\begin{pmatrix} \begin{pmatrix}a^{-1}&bz^{-1}\\0&z^{-1}
\end{pmatrix},&\!\!\!\!\begin{pmatrix}0\\0\end{pmatrix}\end{pmatrix}} da\,db\,dz\\
&+\frac{i \theta}{2\pi}\otimes\int \frac{zb}a \delta_1(a)\,\delta_0(b)\,\delta''_1(z)\,
\lambda_{\begin{pmatrix} \begin{pmatrix}a^{-1}&bz^{-1}\\0&z^{-1}\end{pmatrix},&\!\!\!\!\begin{pmatrix}0\\0\end{pmatrix}\end{pmatrix}}
 da\,db\,dz + O(\theta^2).
\end{align*}
The last  term of order one therefore vanishes, and we get
\begin{equation}
\label{SCL}
\Omega^*_\theta=1\otimes 1
-\frac{i\theta}{2\pi}\, \big(C\otimes Z + X\otimes A- Y\otimes B\big)+ O(\theta^2).
\end{equation}
This means that the Poisson bracket we are quantizing is (up to a scalar) given  by
\begin{multline*}
\{f_1,f_2\}_{P}= \\
(Af_1)(Xf_2) -(Xf_1)(Af_2)+(Zf_1)(Cf_2)-(Cf_1) (Zf_2)-(Bf_1)(Yf_2) +(Yf_1)(Bf_2)
\end{multline*}
for $f_1,f_2\in C^\infty_c(G)$.

\bigskip

\end{document}